\documentclass[%
  onecolumn
   , hidempi
]{mpi2015-cscpreprint}

\usepackage[american]{babel}

\usepackage{graphicx}
\usepackage{amssymb}
\usepackage{amsthm}
\usepackage[mathscr]{eucal}
\usepackage{amsmath}
\usepackage{breqn}

\numberwithin{equation}{section}
\numberwithin{figure}{section}
\numberwithin{table}{section}

\usepackage{import}
\usepackage{xifthen}
\usepackage{pdfpages}
\usepackage{placeins}
\usepackage{transparent}
\usepackage{cleveref}
\usepackage{sidecap}    
\usepackage{floatrow}
\usepackage{todonotes}
\usepackage{algorithm}
\usepackage[noend]{algpseudocode}

\makeatletter
\def\algbackskip{\hskip-\ALG@thistlm}
\makeatother

\usepackage{caption}
\usepackage{subcaption}
\usepackage{arydshln}

\definecolor{lightgray}{gray}{0.9}
\usepackage{color}
\definecolor{bluegreen}{rgb}{0.0, 0.87, 0.87}

\newtheorem{remark}{Remark}

\newtheorem{theorem}{Theorem}
\newcommand*{\bl}[1]{\mathbf{#1}}
\newcommand*{\bldt}[1]{\mathbf{\dot{#1}}}
\newcommand*{\blt}[1]{\mathbf{\tilde{#1}}}
\newcommand*{\bltdt}[1]{\mathbf{\dot{\tilde{#1}}}}
\newcommand*{\blh}[1]{\mathbf{\hat{#1}}}
\newcommand*{\blhdt}[1]{\mathbf{\dot{\hat{#1}}}}
\DeclareMathOperator{\sech}{sech}
\usepackage{mymacros}
\usepackage{todonotes}

\begin{document}
  

\title{Linearly Implicit Global Energy Preserving Reduced-order Models for Cubic Hamiltonian Systems}

\author[$\ast$]{Süleyman Y\i ld\i z}
\affil[$\ast$]{Max Planck Institute for Dynamics of Complex Technical Systems, 39106 Magdeburg, Germany.\authorcr
	\email{yildiz@mpi-magdeburg.mpg.de}, \orcid{0000-0001-7904-605X}
}
  
\author[$\ast\ast$]{Pawan Goyal}
\affil[$\ast\ast$]{Max Planck Institute for Dynamics of Complex Technical Systems, 39106 Magdeburg, Germany.\authorcr
  \email{goyalp@mpi-magdeburg.mpg.de}, \orcid{0000-0003-3072-7780}
}
  
\author[$\dagger\ddagger$]{Peter Benner}
\affil[$\dagger$]{Max Planck Institute for Dynamics of Complex Technical Systems, 39106 Magdeburg, Germany.\authorcr
  \email{benner@mpi-magdeburg.mpg.de}, \orcid{0000-0003-3362-4103}
}
\affil[$\ddagger$]{Otto von Guericke University,  Universit\"atsplatz 2, 39106 Magdeburg, Germany\authorcr
  \email{peter.benner@ovgu.de} 
  \vspace{-0.5cm}
}
  
\shorttitle{Linearly implicit global energy preserving ROMs}
\shortauthor{S. Y\i ld\i z, P. Goyal, P. Benner}
\shortdate{}
  

  
\abstract{%
This work discusses the model reduction problem for large-scale multi-symplectic PDEs with cubic invariants. 
For this, we present a linearly implicit global energy preserving method to construct reduced-order models. 
This allows to construct reduced-order models in the form of Hamiltonian systems suitable for long-time integration. 
Furthermore,  We prove that the constructed reduced-order models preserve global energy, and the spatially discrete equations also preserve the spatially-discrete local energy conversation law. 
We illustrate the efficiency of the proposed method using three numerical examples, namely a linear wave equation, the Korteweg–de Vries equation, and the Camassa-Holm equation, and present a comparison with the classical POD-Galerkin method. 
}

\novelty{
	\begin{enumerate}
		\item A model reduction problem for large-scale multi-symplectic PDEs with cubic invariants is investigated.
		\item Construction of global energy preserving reduced-order models is proposed. 
		\item Several numerical examples are discussed to support our analysis.
\end{enumerate}
} 
\maketitle

\section{Introduction}\label{sec:intro}

Partial differential equations (PDEs) are widely used to study the dynamical behavior of various real-world phenomena such as flow dynamics, weather dynamics, and chemical reactions. In order to capture fine details of dynamical behavior, we require a high-fidelity spatial discretization of the PDEs. This leads to large-scale dynamical systems, which makes the engineering design process (e.g., simulations, optimization, and control) computationally demanding, even sometimes infeasible, particularly when the system is used for multiple queries during the design process. A remedy to the mentioned obstacle can be provided by \emph{model order reduction} (MOR). MOR delivers a tool to construct low-dimensional reduced-order models (ROMs) that approximately capture the high-fidelity or full-order models (FOMs). As a result,  FOMs can be replaced by ROMs, accelerating numerical simulations; thus, engineering design can be done at a modest cost. 

MOR for linear and nonlinear systems has been studied for several decades and applied to various science and engineering applications. We refer to the recent handbook \cite{morBenGQetal21,morBenGQetal21a,morBenGQetal21b} for overviews of various concepts and algorithms used to construct ROMs, as well as wide variety of application domains. In this work, we focus on the proper orthogonal decomposition (POD) approach---arguably one of the most popular techniques to construct ROMs for nonlinear systems, which is constructed using a basis computed using the singular value decomposition of the snapshot or data matrix. 

In this work, we focus on PDEs that can be written in multi-symplectic form. There are numerous examples that fall into this category, e.g., the Korteweg-de Vries (KdV) equation \cite{eidnes2020, gong14}, the Camassa-Holm equation \cite{cohen08ms}, and wave equations \cite{moore2003}. Trajectories of multi-symplectic PDEs fulfill various conversation laws such as local and global energy conservation, local momentum, and multi-symplectic conservation laws. The conservation of these laws is essential in order to obtain accurate long-time horizon simulations. In the past, several numerical integrators have been proposed that preserve properties such as symplectic or multi-symplectic structures, which are often referred to as multi-symplectic integrators. The box/Preissmann scheme \cite{leimkuhler04}, the Euler-box scheme \cite{moore03back},  the Fourier pseudo-spectral collocation scheme \cite{bridges01}, the wavelet collocation method \cite{zhu10}, the average vector field (AVF) method \cite{gong14} and linearly implicit integrators \cite{cai2022, eidnes2020}, are a few examples of such integrators that preserve energy and momentum. A direct application of the POD method often does not yield ROMs that also conserve these physical quantities. As a result, it is commonly observed that ROMs are not stable, and their long-term horizon simulations are often failed. Therefore, we aim to obtain ROMs, preserving the desired physical quantities, and in this paper, we focus on obtaining ROMs that preserve global energy. We also like to highlight some prior work in this direction. The stability of ROMs has been studied in terms of the Lagrangian structure preservation of FOMs in \cite{lall03, carlberg15}. For port-Hamiltonian systems, the passivity and stability of the ROMs are preserved by constructing reduced-order port-Hamiltonian systems in \cite{chaturantabut16, polyuga10}. For Hamiltonian systems, the symplectic structure-preserving ROM is proposed by symplectic transformations in \cite{peng2016}, where it is shown that the ROM inherits the long-term stability of the full-order Hamiltonian models.

In this paper, we study a MOR technique for Hamiltonian PDEs to construct ROMs that preserve the multi-symplectic structure. We construct a FOM by discretizing the PDE using its multi-symplectic form utilizing a finite-difference technique and use Kahan's method \cite{celledoni2012} to integrate in time in order to preserve global energy. We then discuss a structure-preserving technique to construct ROMs using POD Galerkin projection. We show that if the obtained ROM is integrated using the same Kahan's method, it also preserves the physical quantities such as energy. Several numerical experiments also support our analysis. 
Recently, the work  \cite{uzunca2023global} goes in a similar direction, aiming to construct a global energy preserving ROM for multisymplectic PDEs. It utilizes  the global energy conservation scheme inspired by the AVF method in \cite{gong14}. An advantage of this approach is that the global energy is conversed in the ROM, but a disadvantage of the AVF method that it is an implicit method; hence, it is computationally demanding. In this study, we consider the linearly implicit global energy conservation scheme \cite{eidnes2020} to achieve the same goals but more efficiently.

The remainder of the paper is organized as follows. In Section~\ref{sec:FOM}, we provide an overview of the multi-symplectic FOM in space and time, and the energy-preserving Kahan's method. In Section~\ref{sec:ROM}, we present the construction of the multi-symplectic structure-preserving ROMs and show that ROMs converse energy if integrated using Kahan's method. Section~\ref{sec:num} presents several numerical examples and its comparison with the standard POD approach. In Section~\ref{sec:conc}, we present our conclusions and future research directions.

\section{Multi-symplectic Full-order Model} \label{sec:FOM}

Let us consider a quadratic ODE system of the form:
\begin{equation}
	\dot{y} = Q(y)+By+c, \quad y \in \mathbb{R}^N,
	\label{original_quadratic_form}
\end{equation} 
where $B\in \mathbb{R}^{N\times N}$ is a constant symmetric matrix, $ Q (y) \in \mathbb{R}^N$ is a quadratic form, and $c \in \mathbb{R}^{N}$ is a constant vector. As listed in the introduction, there exist several multi-symplectic structures aiming at preserving desired properties such as energy and/or momentum. In this work, we focus on Kahan's methods to integrate the system \eqref{original_quadratic_form} that preserves the global energy of the system. It provides the vector $y$ at the next time step by solving the following nonlinear system of equations:
\begin{equation}\label{eqn:Kahan}
	\dfrac{y^{n+1}-y^n}{\Delta t} =\bar{ Q }(y^n,y^{n+1})+B\left(\frac {y^{n}+y^{n+1}}{2}\right) +c,
\end{equation}
where $y^{n}$ and $y^{n+1}$ are the vectors at time step $n$ and $n+1$, respectively, and the  symmetric bilinear form $ \bar{ Q }(\cdot,\cdot) $ corresponds to the polarization of the quadratic form $ Q $ \cite{celledoni2012}, which is given by
\begin{align*}
	\bar{ Q }(y^{n},y^{n+1}) = \dfrac{1}{2}\left( Q (y^{n}+y^{n+1})- Q (y^{n})- Q \left(y^{n+1}\right)\right).
\end{align*}
Next, we consider a one-dimensional multi-symplectic PDEs of the following form:
\begin{equation}\label{eqn:Ms_PDEs}
	K z_t+L z_x=\nabla_z S(z),	\quad (x,t)\in \mathbb{R}\times \mathbb{R},
\end{equation}
where $ K,  L \in \mathbb{R}^{d\times d} $ are skew-symmetric constant matrices; $ z(x, t)=[z_1(x, t),\ldots, z_d(x, t)]^\top \in\mathbb{R}^{d} $ is the vector of state variables, and $ S(z) : \mathbb{R}^{d} \rightarrow \mathbb{R}$ is a smooth function.
Multi-symplectic PDEs \eqref{eqn:Ms_PDEs} preserve local conservation laws \cite{moore2003se}, which are as follows:
\begin{itemize}
	\item the multi-symplectic conservation law:
	\begin{equation*}
	\partial_t \omega+\partial_x\kappa=0, \quad\omega=dz\wedge K_+dz,\quad\kappa=dz\wedge L_+dz,
	\end{equation*}  
	\item the local energy conservation law:
	\begin{equation*}
	E_t +F_x=0, \quad E=S(z)+z_x^TL_+z,\quad F=-z_t^TL_+z, \quad \text{and}
	\end{equation*}    
	\item  the local momentum conservation law:
	\begin{equation*}
	I_t +G_x=0, \quad G=S(z)+z_t^T K_+z,\quad I=-z_x^T K_+z,
	\end{equation*}    
\end{itemize}   
where $K_+$ and $L_+$ satisfy
$$K=K_+-K_+^\top ,   \quad    L=L_+-L_+^\top ,$$ 
and $ \wedge $ denotes the wedge product.

Let us introduce some notations such as the spatial domain $ \Gamma=\left[a,b\right] $, spatial node $ x_j=a+h(j-1) $, temporal node $ t_n=n\Delta t $, $ j=1,\ldots,N $, $n=0,1,\ldots $, where $ h=\tfrac{b-a}{N} $ is the spatial step size and $ \Delta t $ is the temporal step size. Furthermore, we denote the approximation of the function $ v(x,t) $ at the node $ (x_j,t_n) $ as $ v_j^n $.
Furthermore, for temporal and spatial discretizations, let us define the following difference and averaging operators:
\begin{align*}
	\delta_{t} v_j^n &:= \frac{v_j^{n+1} -v_j^n}{\Delta t}, &\delta_{t}^{1/2} v_j^n &:= \frac{v_j^{n+1} -v_j^{n-1}}{2\Delta t}, & \mu_{t} v_j^n &:= \frac{v_j^{n+1} +v_j^n}{2}, \\
	\delta_{x} v_j^n &:= \frac{v_{j+1}^n -v_{j}^n}{\Delta x}, &  \delta_{x}^{1/2} v_j^n &:= \frac{v_{j+1}^n -v_{j-1}^n}{2\Delta x}, & \mu_{x}v_j^n &:= \frac{v_{j+1}^{n} +v_j^n}{2}.
\end{align*}
The above operators commute \cite{eidnes2020}, i.e.,
$$\delta_{t}^{1/2}\delta_{x}v_j^n=\delta_{x}\delta_{t}^{1/2}v_j^n,\quad \delta_{t}\mu_{x}v_j^n=\mu_{x}\delta_{t}v_j^n,\quad \mu_{t}\delta_{x}^{1/2}v_j^n=\delta_{x}^{1/2}\mu_{t}v_j^n.$$
Moreover, the difference and averaging operators satisfy the discrete Leibniz rule \cite{eidnes2020},
$$\delta_{t}(uv)_j^n=(\varepsilon u_j^{n+1}+(1-\varepsilon )u_j^n)\delta_{t}v_j^n+\delta_{t}u_j^n((1-\varepsilon) v_j^{n+1}+ \varepsilon v_j^{n}),\quad 0\le\varepsilon\le 1.$$
For instance,
\begin{align*}
	\delta_{t}(uv)_j^n&=u_j^{n}\delta_{t}v_j^n+\delta_{t}u_j^nv_j^{n+1},\qquad\text{for}~~ \varepsilon=0,\\
	\delta_{t}(uv)_j^n&=\mu_{t}u_j^{n}\delta_{t}v_j^n+\delta_{t}u_j^n\mu_{t}v_j^{n},\quad\text{for}~~ \varepsilon=\frac{1}{2},\\
	\delta_{t}(uv)_j^n&=u_j^{n+1}\delta_{t}v_j^n+\delta_{t}u_j^nv_j^{n},\qquad\text{for}~~ \varepsilon=1.
\end{align*}
Let us define the spatially discrete state vector as
$$ \bl{z}(t)=[z_{1,1}(t),\ldots,z_{1,N}(t),z_{2,1}(t),\ldots,z_{2,N}(t),z_{d,1}(t),\ldots,z_{d,N}(t)]^\top ,$$
where $ z_{d,j}(t)= z_d(x_{j},t) $, for $ j=1,2,\ldots,N $. Moreover, let us define the vector $  \bl{z}_{m}(t) $, containing the $m^{\texttt{th}}$ node of each state as follows:
\begin{equation}\label{eqn:element}
	\bl{z}_{m}(t)=[{z}_{1,m}(t),{z}_{2,m}(t),\ldots,{z}_{d,m}(t)]^\top .
\end{equation}
The notation in \eqref{eqn:element} and the discrete Leibniz rules play an essential role in the analysis of the conservation laws.
We discretize the partial derivative $\partial_x$ with the central difference operator $\delta_{x}^{1/2}$ and get the spatially discrete FOM of \eqref{eqn:Ms_PDEs} as follows:
\begin{equation}\label{eqn:semi-discrete-element}
	K\partial_t \bl z_m+ L\delta_{x}^{1/2} \bl z_m =\nabla_z S(\bl z_m), \quad m=1,\ldots,N.
\end{equation}

Following the study \cite{cai2022}, we assume that the scalar-valued function $ S(z) $ has the non-homogeneous form
\begin{align}\label{eqn:energy}
	S(z)=z^\top Q(z)z+z^\top Bz+c^\top z+d,
\end{align}
where $6\cdot Q(z)$ is a symmetric matrix that consists of homogeneous linear polynomials, which corresponds to the linear part of the Hessian $\nabla^2 S(z)$; similarly, $2\cdot B$ is the constant symmetric matrix corresponding to the constant part of the Hessian $\nabla^2 S(y)$, and $c, d$ are a constant vector and scalar, respectively.

Using the anti-symmetric matrix $ D_x $, the spatially discrete equations \eqref{eqn:semi-discrete-element} can be written in a compact form:
\begin{equation}\label{eqn:Ms-semi-discrete}
	\bl K\bldt z+ \bl L\bl{D}_x\bl z=\nabla_{\bl z}\bl S(\bl z),
\end{equation}
where $ \bl S(\bl z) : \mathbb{R}^{ N\cdot d} \rightarrow \mathbb{R} $, $ \bl{K}=(K \otimes I_N)\in \mathbb{R}^{N\cdot d\times N\cdot d}$, $ \bl{L}=( L \otimes I_N)\in \mathbb{R}^{N\cdot d\times N\cdot d}$,  $ \bl{D}_x=( I_d\otimes D_x)\in \mathbb{R}^{N\cdot d\times N\cdot d},$  $ I_N\in\mathbb{R}^{N\times N},I_d\in\mathbb{R}^{d\times d} $ are identity matrices, and  $ \otimes $ denotes the Kronecker product.

Then the fully discrete globally energy preserving model can be written by utilizing the polarisation of $ S(z) $ as follows \cite{eidnes2020}:
\begin{equation}\label{eqn:LIGEP-element}
	K\delta_t \bl z_m^n+L\delta_{x}^{1/2}\mu_t\bl z_m^{n} =3\frac {\partial \bar{S}}{\partial x}\bigg\rvert_{(\bl z_m^n,\bl z_m^{n+1})},
\end{equation} 
where the polarized function $ \bar{S}(x,y,z) $ has the following form:
\begin{align}\label{eqn:polarisedH}
	\bar{S}(x,y,z)  = x^\top Q(y)z+\frac {1}{3} (x^\top  B y + y^\top  B z + z^\top  B x)+ \frac{1}{3} c^\top  (x+y+z) + d.
\end{align} 
The polarized function $\bar{S}(x,y,z)$ in \eqref{eqn:polarisedH} satisfies the property $\bar{S}(x,x,x)=S(x)$, and it is symmetric with respect to $ x,y $ and $ z $.
Moreover, the partial derivative of the polarized function $\bar{S}(x,y,z)$ with respect to $ x $ is 
$$\frac{\partial \bar{S}(x,y,z)}{\partial x}=Q(y)z+\frac {B(y+z)}{3}+\frac {c}{3}.$$
Notice that the linearly implicit global energy preserving (LIGEP) method is proposed for a general skew-symmetric differential matrix $ {D}_x $ in \cite{eidnes2020}. Nevertheless, here we will focus only on the scenario where $ {D}_x $ is obtained from the central difference operator.

Finally, the fully discrete equations \eqref{eqn:LIGEP-element} can be written in compact form as follows:
\begin{equation}\label{eqn:LIGEP-FOM}
	\bl K \delta_{t} \bl z^{n}+ \bl L\mu_{t}\bl{D}_x\bl z^{n}=3\frac {\partial \bar{\bl  S}}{\partial\bl x}\bigg\rvert_{(\bl z^{n},\bl z^{n+1})},
\end{equation}  
where $ \bl z^{n}= \bl z(t_n) $ and $\bar{\bl S}(\cdot,\cdot,\cdot):\mathbb{R}^{ N\cdot d}\times\mathbb{R}^{ N\cdot d}\times\mathbb{R}^{ N\cdot d} \rightarrow \mathbb{R}$ is a symmetric $3$-tensor satisfying $\bar{\bl S}(\bl x,\bl x,\bl x) = \bl S(\bl x)$.

\section{Structure-preserving Reduced-order Model}\label{sec:ROM}
Consider the discrete state vector $  \bl z(t_i) \in \mathbb{R}^{d\cdot N} $ for $ i=1,2,\ldots,N_t $. Let us define the following snapshot matrix
\begin{equation}\label{eqn:data_mat}
	\bl S=\left[\bl z(t_1),\ldots,\bl z(t_{N_t})\right]\in \mathbf{R}^{d\cdot N\times N_t}.
\end{equation}
Here, we are interested in the following optimization problem:
\begin{equation}\label{eqn:opt}
	\min_{\substack{ \bl V }} \| \bl S - \bl V\bl V^\top \bl S \|_F, \quad \text{subject to}\quad \bl V^\top \bl K\bl V=\bl K_r,  \ \bl V^\top \bl L\bl V=\bl L_r,
\end{equation}
where $ \|\cdot\|_{F} $ denotes the Frobenius norm; $ \bl{K}_r=(K\otimes I_r)\in \mathbb{R}^{r\cdot d\times r\cdot d}$, $ \bl{L}_r=( L \otimes I_r)\in \mathbb{R}^{r\cdot d\times r\cdot d}$ with $ I_r \in \mathbb{R}^{r\times r} $ being the identity matrix.

One possible solution of the optimization problem \eqref{eqn:opt} may be derived by using an idea similar to the \emph{cotangent lift} in the symplectic MOR \cite{peng2016}. To demonstrate the procedure, let us define global snaphot matrix of the form
\begin{equation}\label{eqn:global_snap}
	\bl Z=[\bl Z_1,\ldots, \bl Z_d]\in \mathbb{R}^{N\times d\cdot N_t}
\end{equation}
where $ N_t $ is the number of time steps, and
$$ \bl Z_i=\left[[z_{i,1}(t_1),\ldots,z_{i,N}(t_1)]^\top,\ldots, [z_{i,1}(t_{N_t}),\ldots,z_{i,N}(t_{N_t})]^\top \right]\in \mathbb{R}^{N\times N_t}, \quad i=1,\ldots,d. $$

Assuming $ V\in \mathbb{R}^{N\times r} $ is the POD basis of the global snapshot matrix \eqref{eqn:global_snap}, we can construct the basis as $ \bl{V}=( I_d\otimes V)\in \mathbb{R}^{d\cdot N\times d\cdot r}  $, where $ I_d\in \mathbb{R}^{d\times d} $ is the identity matrix. 
Since the matrix $\bl V$ is an orthogonal matrix, i.e., $ \bl{V}^\top \bl{V}=I_{d\cdot r } $, we can easily show that the following properties hold:
\begin{align}\label{eq:reduced_relation}
	\bl V^\top \bl K\bl V &= \bl K_r,  & \bl V^\top \bl L\bl V&= \bl L_r, &
	\bl{V}^\top\bl{K}&= \bl{K}_r\bl{V}^\top, & \bl{V}^\top\bl{L}&=  \bl{L}_r\bl{V}^\top.
\end{align}
Using the definition in \eqref{eqn:data_mat}, the snapshot matrix $ \bl S $ can be written as $ \bl S=\left[\bl Z_1^\top,\ldots,\bl Z_d^\top\right]^\top $. Now, let us consider the objective function of \eqref{eqn:opt},
\begin{equation}\label{eqn:opt_basis}
	\begin{split}
		\| \bl S - \bl V\bl V^\top \bl S \|_F&=\| \bl S - \left( I_d\otimes \left(VV^\top\right)\right) \bl S \|_F\\
		&=\left\| \left[\left(I_N-VV^\top\right)\bl Z_1,\ldots,\left(I_N-VV^\top\right)\bl Z_d\right] \right\|_F\\
		&=\left\| \left(I_N-VV^\top\right)\left[\bl Z_1,\ldots,\bl Z_d\right] \right\|_F\\
		&=\left\| \left(I_N-VV^\top\right)\bl Z \right\|_F.
	\end{split}
\end{equation}
Equation \eqref{eqn:opt_basis} shows that $ \bl V $ is a minimizer of the optimization problem \eqref{eqn:opt} . Hence, by the approximation $ \bl z\approx \blh{z}=\bl{V}\blt z$, we obtain the following:
\begin{equation}\label{FOM-approx}
	\bl{K} \bl{V}\bltdt z+\bl{L}\bl{D}_x\bl{V}\blt z=\nabla_{\bl z}\bl S( \bl{V}\blt z) + \cR(\blt z),
\end{equation}
where $\cR(\blt z)$ is the residual. Projecting \eqref{FOM-approx} from the left-hand side using $V^\top$ and assuming $\cR(\blt z)$ is orthonormal to $V$, the ROM can be written as follows:
\begin{equation}\label{eqn:Ms-proj}
	\bl{V}^\top\bl{K} \bl{V}\bltdt z+\bl{V}^\top\bl{L}\bl{D}_x\bl{V}\blt z=\bl{V}^\top\nabla_{\bl z}\bl S( \bl{V}\blt z).	
\end{equation} 
Using the relations in \eqref{eq:reduced_relation}, we have
\begin{equation}
	\bl{K}_r\bltdt z+\bl{L}_r\bl{V}^\top\bl{D}_x\bl{V}\blt z=\bl{V}^\top\nabla_{\bl z}\bl S(\bl{V}\blt z),
\end{equation} 
which can be written as follow:
\begin{equation}\label{MS-sd-rom}
	\bl{K}_r\bltdt z+\bl{L}_r\blt{D}_x\blt z=\bl{V}^\top\nabla_{\bl z}\bl S(\bl{V}\blt z),
\end{equation} 
where $ \blt{D}_x= \bl V^\top\bl{D}_x \bl V $. Notice that the semi-discrete ROM \eqref{MS-sd-rom} has a similar multi-symplectic structure as in \eqref{eqn:Ms_PDEs}. It can be shown that $ \bl{K}_r $ and $ \bl{L}_r $ are skew-symmetric.

\begin{remark}
Since in this study we focus only on cubic Hamiltonian systems, the resulting Hamiltonian PDEs possess quadratic nonlinearities. Therefore, the online cost of the LIGEP-ROM \eqref{MS-sd-rom} can be further reduced by using tensor algebra. Hence, in the online stage, the computational cost of the global energy preserving ROM depends only on the reduced dimension $r$. However, in the offline stage, the computation of the basis becomes more expensive compared to the Galerkin POD model due to the amount of $d$ auxiliary snapshots needed in \eqref{eqn:global_snap}. 
\end{remark}

In the following, we do a similar analysis as in \cite[Proposition 3.8]{moore2003} to show that the spatially discrete ROM \eqref{eqn:Ms-proj} satisfies the semi-discrete energy conservation law.
\begin{theorem}
	The spatially-discrete equation \eqref{eqn:Ms-proj} yields a semi-discrete energy conservation law
	\begin{equation*}
		\partial_tE_m+\delta_{x}(F_{m})=0,
	\end{equation*}
	with
	\begin{align*}
		E_m&=S( \blh{z}_m)-\frac{1}{2}\langle  \blh{z}_{m},L\delta_x^{1/2}\blh z_m\rangle\\
		F_{m}&=\frac{1}{4}\left[\langle \blh{z}_{m}, L \blhdt{z}_{m-1}\rangle+\langle \blh{z}_{m-1}, L \blhdt{z}_{m}\rangle\right].
	\end{align*}
\end{theorem}
\begin{proof}
	Consider the semi-discrete equation \eqref{eqn:Ms-proj}
	$$\bl V^\top\bl K\bl V\bltdt z+ \bl{V}^\top\bl{L}\bl{D}_x\bl{V}\blt z=\bl{V}^\top\nabla_{\bl z}\bl S(\blh z).$$
	By taking the inner product of both sides with $ \bltdt z$, we  obtain
	$$\langle \bltdt{z},\bl V^\top\bl K\bl V\bltdt z\rangle+ \langle \bltdt{z},\bl V^\top\bl L\bl{D}_x\blh z\rangle=\langle \bltdt{z},\bl{V}^\top\nabla_{\bl z}\bl S(\blh z)\rangle,	$$
	and by using the notation in \eqref{eqn:element}, we can write the above equation as
	\begin{equation*}
		\langle \blhdt{z}_{m},K\blhdt z_{m}\rangle+ \langle \blhdt{z}_{m},L\delta_x^{1/2}\blh z_m\rangle=\langle \blhdt{z}_{m},\nabla_zS(\blh z_m)\rangle, \qquad  m=1,\ldots, N.
	\end{equation*}
	From the skew-symmetry property of the matrix $K$, we have $ \langle \blhdt{z}_{m},K\blhdt z_{m}\rangle=0 $. Additionally, using the identity $ \partial_t S(\blh z_m)=\langle \blhdt{z}_{m},\nabla_zS(\blh z_m)\rangle $, we have
	\begin{align*}
		\partial_t S(\blh z_m)&= \langle \blhdt{z}_{m},L\delta_x^{1/2}\blh z_m\rangle\\
		&=\frac{1}{2}\left( \partial_t\langle \blh{z}_{m},L\delta_x^{1/2}\blh z_m\rangle+\langle \blhdt{z}_{m},L\delta_x^{1/2}\blh z_m\rangle- \langle \blh{z}_{m},L\delta_x^{1/2}\blhdt z_m\rangle\right).
	\end{align*}
	Moreover, we have the following relation:
	\begin{align*}
		&\hspace{-2cm}\langle \blhdt{z}_{m},L\delta_x^{1/2}\blh z_m\rangle- \langle \blh{z}_{m},L\delta_x^{1/2}\blhdt z_m\rangle\\
		&=\frac{1}{2h}\left(\langle \blhdt{z}_{m},L(\blh z_{m+1}-\blh z_{m-1})\rangle- \langle \blh{z}_{m},L(\blhdt z_{m+1}-\blhdt z_{m-1})\rangle \right)\\
		&=\frac{1}{2h}\left(-\langle \blh{z}_{m+1},L\blhdt{z}_{m}\rangle+\langle \blh{z}_{m-1},L\blhdt{z}_{m}\rangle-\langle \blh{z}_{m},L\blhdt{z}_{m+1}\rangle+\langle \blh{z}_{m},L\blhdt{z}_{m-1}\rangle \right)\\
		&=-\frac{1}{2}\delta_{x}\left(\langle \blh{z}_{m},L\blhdt{z}_{m-1}\rangle+\langle \blh{z}_{m-1},L\blhdt{z}_{m}\rangle\right),
	\end{align*} 
which concludes the proof. 
\end{proof}
Next, we note the fully discrete equations by employing the LIGEP method \cite{eidnes2020} to \eqref{MS-sd-rom}, which can then be written as 
\begin{equation}\label{eqn:LIGEP-ROM}
	\bl V^\top\bl K\delta_{t}\blh z^{n}+ \bl V^\top\bl L\mu_{t}\bl{D}_x\blh z^{n}=3\bl V^\top\frac {\partial \bar{\bl S}}{\partial \bl x}\bigg\rvert_{(\blh z^{n},\blh z^{n+1})}.
\end{equation}  

In the following theorem, we use a some similar analysis with \cite[Theorem 4.5]{eidnes2020} to show that our scheme satisfies global energy conversation law.

\begin{theorem}\label{th:globalenergylaw}
	For periodic boundary conditions $z(x+P,t)=z(x,t)$, the scheme \eqref{eqn:LIGEP-ROM} satisfies the discrete global energy conservation law
	\begin{equation}
		\bar{\mathcal{E}}_r^{n+1}= \bar{\mathcal{E}}_r^{n}, \quad \bar{\mathcal{E}}_r^{n}:=\Delta x\sum_{m=1}^{N} \bar{E}_{m}^n, \quad \Delta x = P/N,
		\label{eq:polglobenergy}
	\end{equation}
	where 
	$$	\bar{E}_m^n = \bar{S}(\blh z_m^n,\blh z_m^n,\blh z_m^{n+1})+ \frac{1}{3}\big((\delta_x^{1/2}\blh z_m^{n})^\top L_+\blh z_m^n+(\delta_x^{1/2} \blh z_m^{n})^\top L_+\blh z_m^{n+1}+(\delta_x^{1/2} \blh z_m^{n+1})^\top L_+\blh z_m^n\big).\\$$
\end{theorem}   
\begin{proof}
	Taking the inner product with $\frac{1}{3}\delta_t \blt z^n$ on both sides of \eqref{eqn:LIGEP-ROM}, we get
	\begin{align}
		\frac{1}{3}(\delta_t \blt z^n)^\top\bl V^\top\bl K\delta_{t}\blh z^{n}+ \frac{1}{3}(\delta_t \blt z^n)^\top\bl V^\top\bl L\mu_{t}\bl{D}_x\blh z^{n}=(\delta_t \blt z^n)^\top\bl V^\top\frac {\partial \bar{\bl S}}{\partial \bl x}\bigg\rvert_{(\blh z^{n},\blh z^{n+1})}
	\end{align}
	which by using the skew-symmetry property of the matrix $K$  can be written  as
	\begin{align}\label{inner product1}
		\frac{1}{3}(\delta_t \blh z_m^n)^\top L\mu_{t}\delta_x^{1/2}\blh z_m^{n}=(\delta_t \blh z_m^n)^\top \frac {\partial \bar{S}}{\partial x}\bigg\rvert_{(\blh z_m^{n},\blh z_m^{n+1})}.
	\end{align}
	Similarly, taking the inner product with $\frac{1}{3}\delta_t \blt z^{n+1}$ on both sides of \eqref{eqn:LIGEP-ROM}, we get
	\begin{align}\label{inner product2}
		\frac{1}{3}(\delta_t \blh z_m^{n+1})^\top K\delta_{t}\blh z_m^{n}+ \frac{1}{3}(\delta_t \blh z_m^{n+1})^\top L\mu_{t}\delta_x^{1/2}\blh z_m^{n}=(\delta_t \blh z_m^{n+1})^\top \frac {\partial \bar{S}}{\partial x}\bigg\rvert_{(\blh z_m^{n},\blh z_m^{n+1})}.
	\end{align} 
	Furthermore, taking the inner product with $\frac{1}{3}\delta_t \blt z^n$ on both sides of \eqref{eqn:LIGEP-ROM} for the next time step, we have
	\begin{align}\label{inner product3}
		\frac{1}{3}(\delta_t \blh z_m^{n})^\top K\delta_{t}\blh z_m^{n+1}+ \frac{1}{3}(\delta_t \blh z_m^{n})^\top L\mu_{t}\delta_x^{1/2}\blh z_m^{n+1}=(\delta_t \blh z_m^{n})^\top \frac {\partial \bar{S}}{\partial x}\bigg\rvert_{(\blh z_m^{n+1},\blh z_m^{n+2})}.
	\end{align} 
	Adding equations \eqref{inner product1}, \eqref{inner product2} and \eqref{inner product3}, we get 
	\begin{equation}\label{inner product}
		\begin{split}
			\frac{1}{3} \big((\delta_t \blh z_m^n)^\top L\mu_{t}\delta_x^{1/2}\blh z_m^{n} +(\delta_t \blh z_m^{n+1})^\top L\mu_{t}\delta_x^{1/2}\blh z_m^{n}&\\
			+ (\delta_t \blh z_m^{n})^\top L\mu_{t}\delta_x^{1/2}\blh z_m^{n+1}\big)& =\delta_t \bar{S}(\blh z_m^n,\blh z_m^n,\blh z_m^{n+1}),
		\end{split}
	\end{equation}
	and by using the commutative laws and discrete Leibniz rules, we have
	\begin{equation}\label{using discrete Leibnize rule}
		\begin{split}
			\delta_t((\delta_x^{1/2}\blh z^{n}_m)^\top L_+\blh z_m^n)&=(\delta_x^{1/2}\delta_t\blh z^{n}_m)^\top  L_+\mu_t\blh z_m^n+(\delta_x^{1/2}\mu_t \blh z_m^{n})^\top L_+ \delta_t \blh z_m^n,\\
			\delta_t((\delta_x^{1/2} \blh z_m^{n})^\top L_+\blh z_m^{n+1})&=(\delta_x^{1/2} \delta_t\blh z_m^{n})^\top  L_+\mu_t\blh z_m^{n+1}+(\delta_x^{1/2}\mu_t\blh z_m^{n})^\top L_+ \delta_t \blh z_m^{n+1},\\
			\delta_t((\delta_x^{1/2}\blh z_m^{n+1})^\top L_+\blh z_m^n)&=(\delta_x^{1/2}\delta_t\blh z_m^{n+1})^\top  L_+\mu_t\blh z_m^n+(\delta_x^{1/2}\mu_t\blh z_m^{n+1})^\top L_+ \delta_t \blh z_m^n.
		\end{split} 
	\end{equation} 
	Based on the above equations \eqref{inner product} and \eqref{using discrete Leibnize rule}, we obtain
	\begin{equation*}
		\begin{split}
			\delta_t \bar{E}_m^n =& \,  \delta_t \bar{S}(\blh z_m^n,\blh z_m^n,\blh z_m^{n+1})+ \frac{1}{3}\delta_t\big(((\delta_x^{1/2}\blh z_m^{n})^\top L_+\blh z_m^n)+(\delta_x^{1/2} \blh z_m^{n})^\top L_+\blh z_m^{n+1}+(\delta_x^{1/2} \blh z_m^{n+1})^\top L_+\blh z_m^n\big)\\
			= & \, \frac{1}{3}\big((\delta_t\blh z_m^n)^\top L_+(\delta_x^{1/2}\mu_t\blh z_m^{n})+(\delta_x^{1/2} \delta_t\blh z_m^{n})^\top L_+\mu_t\blh z_m^n\big)\\
			& \, +\frac{1}{3}\big((\delta_t\blh z_m^{n+1})^\top L_+(\delta_x^{1/2} \mu_t\blh z_m^{n})+(\delta_x^{1/2} \delta_t\blh z_m^{n+1})^\top L_+\mu_t\blh z_m^n\big)\\
			& \, +\frac{1}{3}\big((\delta_t\blh z_m^{n})^\top L_+(\delta_x^{1/2}\mu_t\blh z_m^{n+1})+(\delta_x^{1/2} \delta_t\blh z_m^{n})^\top L_+\mu_t\blh z_m^{n+1}\big)\\
			=  & \, \frac{1}{6h}\big((\delta_t\blh z_m^n)^\top L_+(\mu_t\blh z_{m+1}^{n}-\mu_t\blh z_{m-1}^{n})+(\delta_t\blh z_{m+1}^{n}-\delta_t\blh z_{m-1}^{n})^\top L_+\mu_t\blh z_m^n\big)\\
			& \, +\frac{1}{6h}\big((\delta_t\blh z_m^{n+1})^\top L_+(\mu_t\blh z_{m+1}^{n}-\mu_t\blh z_{m-1}^{n})+(\delta_t\blh z_{m+1}^{n+1}-\delta_t\blh z_{m-1}^{n+1})^\top L_+\mu_t\blh z_m^n\big)\\
			& \, +\frac{1}{6h}\big((\delta_t\blh z_m^{n})^\top L_+(\mu_t\blh z_{m+1}^{n+1}-\mu_t\blh z_{m-1}^{n+1})+(\delta_t\blh z_{m+1}^{n}-\delta_t\blh z_{m-1}^{n})^\top L_+\mu_t\blh z_m^{n+1}\big)\\
		\end{split}.
	\end{equation*}
	which by using periodic boundary conditions implies 
	\begin{equation*}
		\begin{split}
			\sum_{m=1}^{N}\delta_t \bar{E}_m^n
			=&\frac{1}{6h}\sum_{m=1}^{N}\big((\delta_t\blh z_m^n)^\top L_+(\mu_t\blh z_{m+1}^{n})-(\delta_t\blh z_{m-1}^{n})^\top L_+\mu_t\blh z_m^n\big)\\
			& \, +\frac{1}{6h}\sum_{m=1}^{N}\big( (\delta_t\blh z_{m+1}^{n})^\top L_+\mu_t\blh z_m^n -(\delta_t\blh z_m^n)^\top L_+(\mu_t\blh z_{m-1}^{n}) \big)\\
			& \, +\frac{1}{6h}\sum_{m=1}^{N}\big((\delta_t\blh z_m^{n+1})^\top L_+(\mu_t\blh z_{m+1}^{n})-(\delta_t\blh z_{m-1}^{n+1})^\top L_+\mu_t\blh z_m^n\big)\\
			& \, +\frac{1}{6h}\sum_{m=1}^{N}\big((\delta_t\blh z_{m+1}^{n+1})^\top L_+\mu_t\blh z_m^n-(\delta_t\blh z_m^{n+1})^\top L_+(\mu_t\blh z_{m-1}^{n})\big)\\
			& \, +\frac{1}{6h}\sum_{m=1}^{N}\big((\delta_t\blh z_m^{n})^\top L_+(\mu_t\blh z_{m+1}^{n+1})-(\delta_t\blh z_{m-1}^{n})^\top L_+\mu_t\blh z_m^{n+1}\big)\\
			& \, +\frac{1}{6h}\sum_{m=1}^{N}\big((\delta_t\blh z_{m+1}^{n})^\top L_+\mu_t\blh z_m^{n+1}-(\delta_t\blh z_m^{n})^\top L_+(\mu_t\blh z_{m-1}^{n+1})\big)=0.
		\end{split}
	\end{equation*}
	Hence,  the discrete global energy conservation law	$\bar{\mathcal{E}}_r^{n+1}= \bar{\mathcal{E}}_r^{n}$ is satisfied.
\end{proof}

\section{Numerical Results}\label{sec:num}	
In this section, we demonstrate the performance of the proposed global energy preserving method to construct ROMs using three examples. We compare the LIGEP-ROM \eqref{eqn:LIGEP-ROM} with the classical POD-Galerkin model, see, e.g., \cite{berkooz93,morVol13}. Our first example represents a linear wave equation with a quadratic Hamiltonian function. Then, we study the Korteweg--de Vries (KdV) equation to validate our results. Finally, we test our method using the Camassa-Holm (CH) equation. Both KdV and CH equations have cubic Hamiltonians.
We test the accuracy of the models with the following relative state error
 \begin{equation}\label{eqn:state-l2err}
\frac{ \|\bl u(t_j)-\blh u(t_j)\|_2 }{\|\bl u(t_j)\|_2},
 \end{equation}
 where $ \bl u(t_j) $ is the full-order solution, and $ \blh u(t_j) $ is the approximated solution either obtained using the LIGEP-ROM or using the POD-Galerkin method.

Let us denote the polarized global energy preserved by FOM \eqref{eqn:LIGEP-FOM} with $ \bar{\mathcal{E}} $. We use two different global energy error measures to examine the performance of the ROMs. First, we test the global energy preservation performance of the FOM and ROMs separately using the following absolute error:
   \begin{equation}\label{eqn:pole_err}
 | \bar{\mathcal{E}}(t_j)-\bar{\mathcal{E}}(t_0)|,
 \end{equation}
 where for the ROMs, instead of full-order polarized energy $  \bar{\mathcal{E}} $, we use the approximated polarized energy $ \bar{\mathcal{E}}_r $. The global energy error \eqref{eqn:pole_err} tests the global energy preservation performance individually. Hence, preserving the approximated global energy in terms of \eqref{eqn:pole_err} does not necessarily imply the convergence of the approximated global energies to the full-order global energy,  so to check the error between the approximated global energy  $ \bar{\mathcal{E}}_r $ and full-order global energy  $ \bar{\mathcal{E}}$, we use the following absolute error
 \begin{equation}\label{eqn:pole_fr_err}
 | \bar{\mathcal{E}}(t_n)-\bar{\mathcal{E}}_r(t_n)|.
 \end{equation}
\subsection{Linear wave equation}\label{subsec:wave}
We begin with the simple linear wave equation of the form:
\begin{equation}\label{eqn:wave}
	u_{tt} =u_{xx},
\end{equation}
which is an example of multi-symplectic Hamiltonian PDEs \eqref{eqn:Ms_PDEs}. Introducing the following variables
\begin{equation*}
	\begin{split}
		v=u_t, \quad w=u_x
	\end{split},
\end{equation*}
the wave equation \eqref{eqn:wave} can be written as a multi-symplectic PDE with
\begin{equation*}
	z=
	\begin{bmatrix}
		u\\
		v\\
		w
	\end{bmatrix},\qquad	
	K=
	\begin{bmatrix}
		0  &-1 & 0  \\
		1  & 0 & 0  \\
		0  & 0 & 0 
	\end{bmatrix},\qquad
	L=
	\begin{bmatrix}
		0  & 0  & 1 \\
		0  & 0  & 0 \\
		-1  & 0  & 0   
	\end{bmatrix}
\end{equation*}	
and the Hamiltonian $S(z)=\frac{1}{2}(v^2-w^2)$. The fully-discrete wave equation obtained with the LIGEP method \eqref{eqn:LIGEP-FOM} reads as follows:
\begin{equation}\label{eqn:wave-LIGEP-couple}
	\begin{split}
		-\delta_t v_j^n+ \delta_x^{1/2}\mu_tw_j^n &=0,\\
		\delta_tu_j^n &= \mu_tv_j^n,\\
		-\delta_x^{1/2} \mu_t u_j^n &=-\mu_tw_j^n.
	\end{split}
\end{equation}
After the elimination of the auxiliary variables, fully discrete equations can be  equivalently written as follow:
\begin{align}\label{eqn:wave_LIGEP}
	\delta_t^2u_j^n - \mu_t^2 \left(\delta_x^{1/2}\right)^2 u^{n}_j=0,
\end{align}
where $\left(\delta_x^{1/2}\right)^2=\delta_x^{1/2}\delta_x^{1/2}$. The polarized discrete energy preserved by the FOM \eqref{eqn:wave_LIGEP} is given by
\begin{equation}\label{eqn:wave-fom-pol-ener}
	\begin{split}
		\bar{\mathcal{E}}(t_n) = & \, \frac{\Delta x}{6}\sum_{j=1}^{N}\Big(2(\delta_x^{1/2}u^n_j)(\delta_x^{1/2} u^{n+1}_j) + (\delta_x^{1/2}u^n_j)^2+2v^n_jv^{n+1}_j + (v^n_j)^2\Big),
	\end{split}
\end{equation}
where $ v^n_j=\delta_tu^n_j-\frac{\Delta t}{2}\mu_t(\delta_x^{1/2})^{2}u^n_j .$

In order to employ the LIGEP-ROM, we  construct the following snapshot matrix:
$$\bl Z=\left[ \bl u(t_1),\ldots,\bl u(t_{N_t}),\bl v(t_1),\ldots,\bl v(t_{N_t}),\bl w(t_1),\ldots,\bl w(t_{N_t})\right]\in \mathbb{R}^{N\times 3 N_t},$$
where the states $ \bl v $ and $ \bl w $ can be obtained from \eqref{eqn:wave-LIGEP-couple}, which are $ v^n_j=\delta_tu^n_j-\frac{\Delta t}{2}\mu_t(\delta_x^{1/2})^{2}u^n_j $ and $ \delta_x^{1/2}  u_j^n =w_j^n .$  
Therefore, a fully-discrete ROM for the wave equation by the LIGEP-ROM \eqref{eqn:LIGEP-ROM} method can be written as follow:
\begin{equation}\label{eqn:wave_LIGEP_ROM_el}
	\begin{split}
		-\delta_t \blt v^n+ \mu_t \tilde{D}_x\blt w^n &=0,\\
		\delta_t \blt u^n &= \mu_t \blt v^n,\\
		-\mu_t \tilde{D}_x \blt u^n &=-\mu_t \blt w^n,
	\end{split}
\end{equation}
where $ \tilde{D}_x=V^\top D_xV $.
Eliminating  the auxiliary variables in \eqref{eqn:wave_LIGEP_ROM_el}, the fully-discrete LIGEP-ROM can be written as
\begin{equation}\label{eqn:wave_LIGEP_ROM}
	\delta_t^2\blt u^n - \mu_t^2 \tilde{D}_x^2 \blt u^{n}=0.
\end{equation}

The ROM \eqref{eqn:wave_LIGEP_ROM} preserves the following approximated polarised energy
\begin{equation}\label{eqn:wave-rom-pol-ener}
	\begin{split}
		\bar{\mathcal{E}_r}(t_n) = & \, \frac{\Delta x}{6}\sum_{j=1}^{N}\Big(2( V \tilde{D}_x \blt u^n)_j( V \tilde{D}_x \blt u^{n+1})_j + ( V \tilde{D}_x \blt u^n)_j^2+2\blh v^n_j \blh v^{n+1}_j + (\blh v^n_j)^2\Big),
	\end{split}
\end{equation}
where 
\begin{equation*}
	\blh v^n_j=\delta_t \blh u^n_j-\frac{\Delta t}{2}\mu_t\left(( V \tilde{D}_x)^{2}\blt u^n\right)_j.
\end{equation*}

To compare the LIGEP-ROM with the classical POD-Galerkin method, we consider the Hamiltonian form of the wave equation \cite{peng2016},
\begin{equation}\label{eqn:Ham_wave}
	\frac{\partial}{\partial t}
	\begin{bmatrix}
		u\\v
	\end{bmatrix}
	=
	\begin{bmatrix}
		0&1\\\partial_{xx}&0
	\end{bmatrix}
	\begin{bmatrix}
		u\\v
	\end{bmatrix},
\end{equation}
which is obtained by introducing the variable $v=u_t$.
After spatial discretization, a semi-discrete Hamiltonian ODE system for wave equation reads as 
\begin{equation}\label{eqn:Wave-Ham-ODE}
	\frac{d \bl y}{d t}=\bl J \bl y, 
\end{equation}
where
$$\bl y=\begin{bmatrix}
	\bl u\\ \bl v
\end{bmatrix}, \quad \bl{J}=\begin{bmatrix}
	0_N&I_N\\ D_{xx}& 0_N
\end{bmatrix}, $$
$ D_{xx}\in \mathbb{R}^{N\times N}$ is the three-point central difference approximation of $\partial_{xx} $ and $0_N\in \mathbb{R}^{N\times N}$ is null matrix.

We obtain the snapshots for the classical POD-Galerkin method by discretizing the Hamiltonian ODE system \eqref{eqn:Wave-Ham-ODE} by Kahan's method \eqref{eqn:Kahan} in time. The snapshot matrix for the classical POD method constructed as follows:
$$\bl S=\left[\bl y(t_1),\ldots,\bl y(t_{N_t})\right]\in \mathbf{R}^{2 N\times N_t} .$$

For a clear distinction, let us denote $ W \in \mathbb{R}^{2N\times r} $ as the POD matrix, then the semi-discrete POD-Galerkin model is computed as follows:
\begin{equation}\label{eqn:Wave-ROM-ODE}
	\frac{d \blt y}{d t}=\blt J \blt y,
\end{equation}
where $ \blt J=W^\top\bl JW\in \mathbb{R}^{r\times r} $. We obtain a fully-discrete POD-Galerkin model by integrating the semi-discrete POD-Galerkin model \eqref{eqn:Wave-ROM-ODE} with Kahan's method \eqref{eqn:Kahan} in time.

We consider the following initial conditions \cite{kosmas2021}
\begin{align*}
	u_t(x,0)&=0,\\
	u(x,0)&=\sech(x),
\end{align*}
on the domain $ [-10,10] $ with periodic boundary conditions.
To examine the stability of the ROMs, we set the final time for basis construction for both ROMs to $T=10$, and then we simulate the ROMs up to $T=40$. We set the temporal step-size as $\Delta t=0.01$  and spatial step-size $\Delta x=0.02$.

In Figure~\ref{fig:wave-l2}, we plot the relative state error \eqref{eqn:state-l2err} of the ROMs, which shows that by increasing the order of the ROMs, POD-Galerkin does not yield a stable surrogate model after time $T=10$, whereas when global energy preservation is enforced to the ROM, the model remains stable. Additionally, we present the full and reduced-order solutions in Figure~\ref{fig:wave-fmrms} for reduced-order $r=50$, which shows that the important characteristic of the FOM is captured in the LIGEP-ROM case. 

\begin{figure}[tb]
	\centering
	\includegraphics[width=0.48\linewidth]{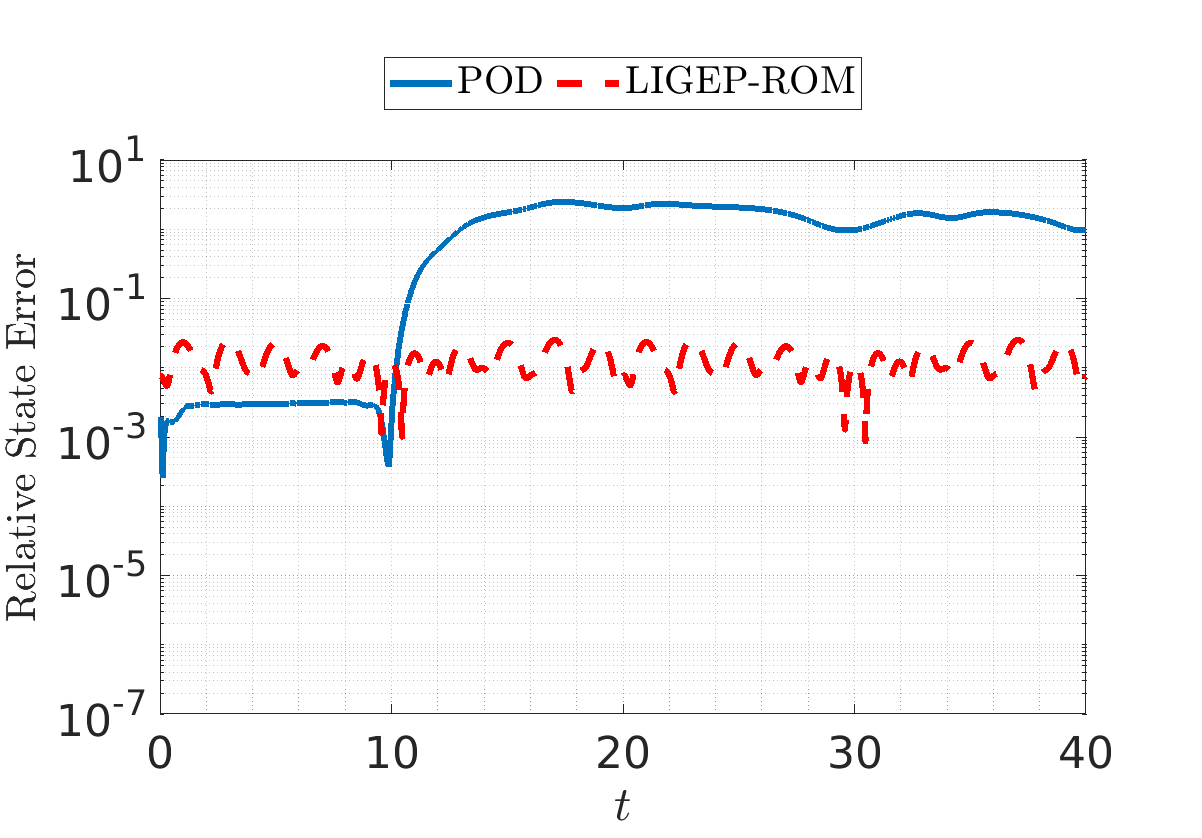}	
	\includegraphics[width=0.48\linewidth]{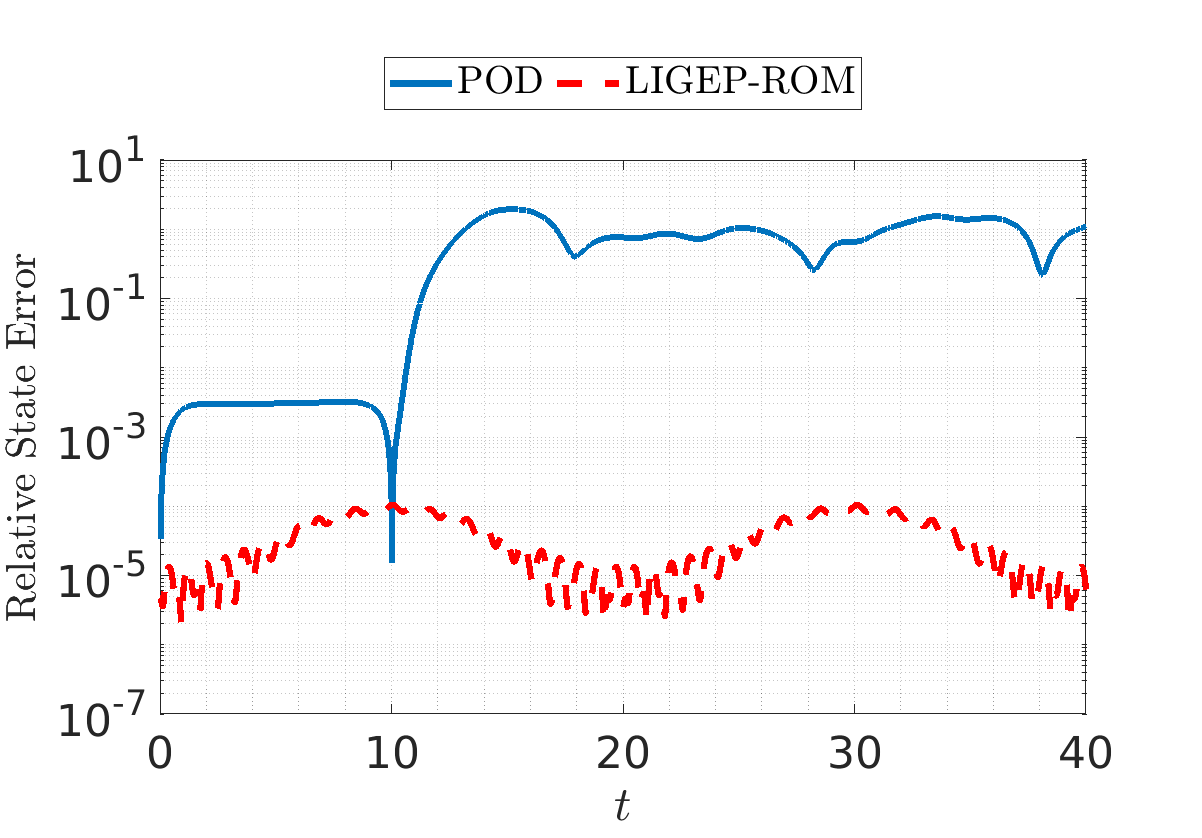}
	
	\caption{Linear wave equation: relative state errors for ROMs of orders $r=20$ and $ r=50 $ are shown in the left and right, respectively.}
	\label{fig:wave-l2}
\end{figure}

\begin{figure}[tb]
	\centering
	\includegraphics[width=0.32\linewidth]{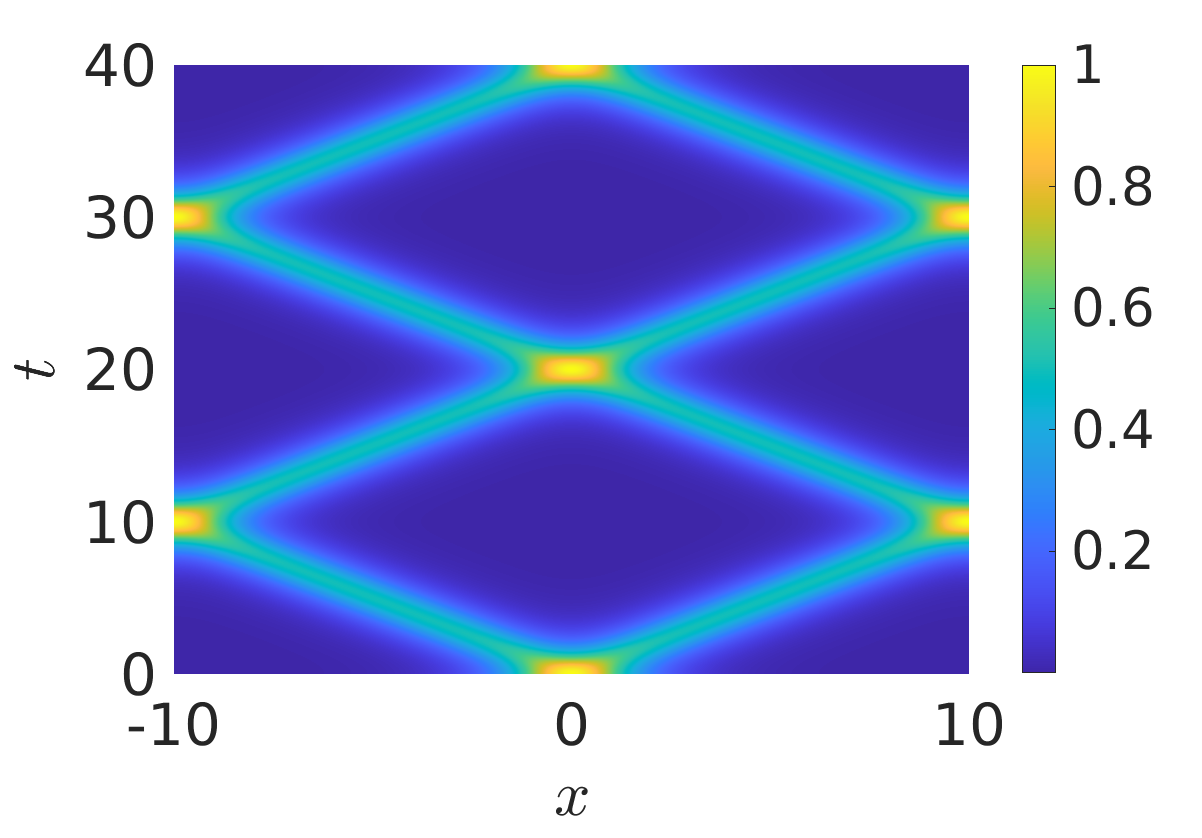} 
	\includegraphics[width=0.32\linewidth]{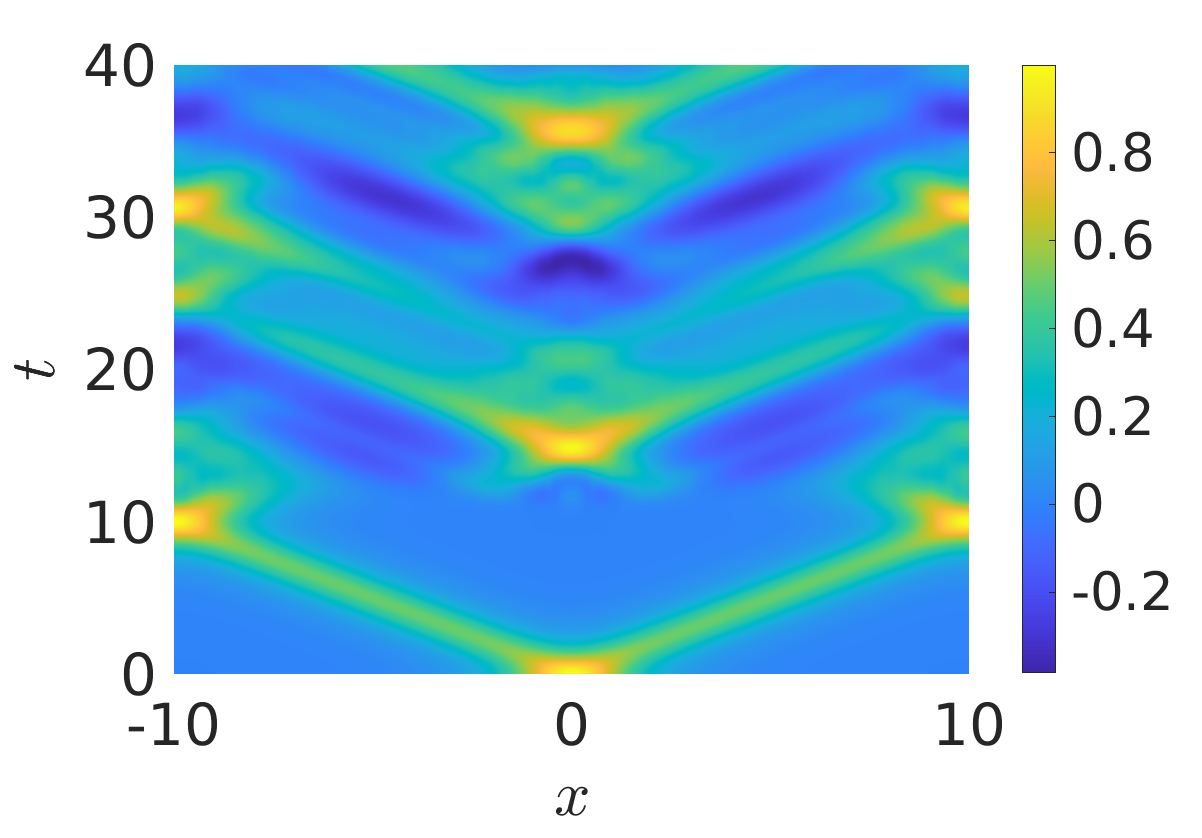} 
	\includegraphics[width=0.32\linewidth]{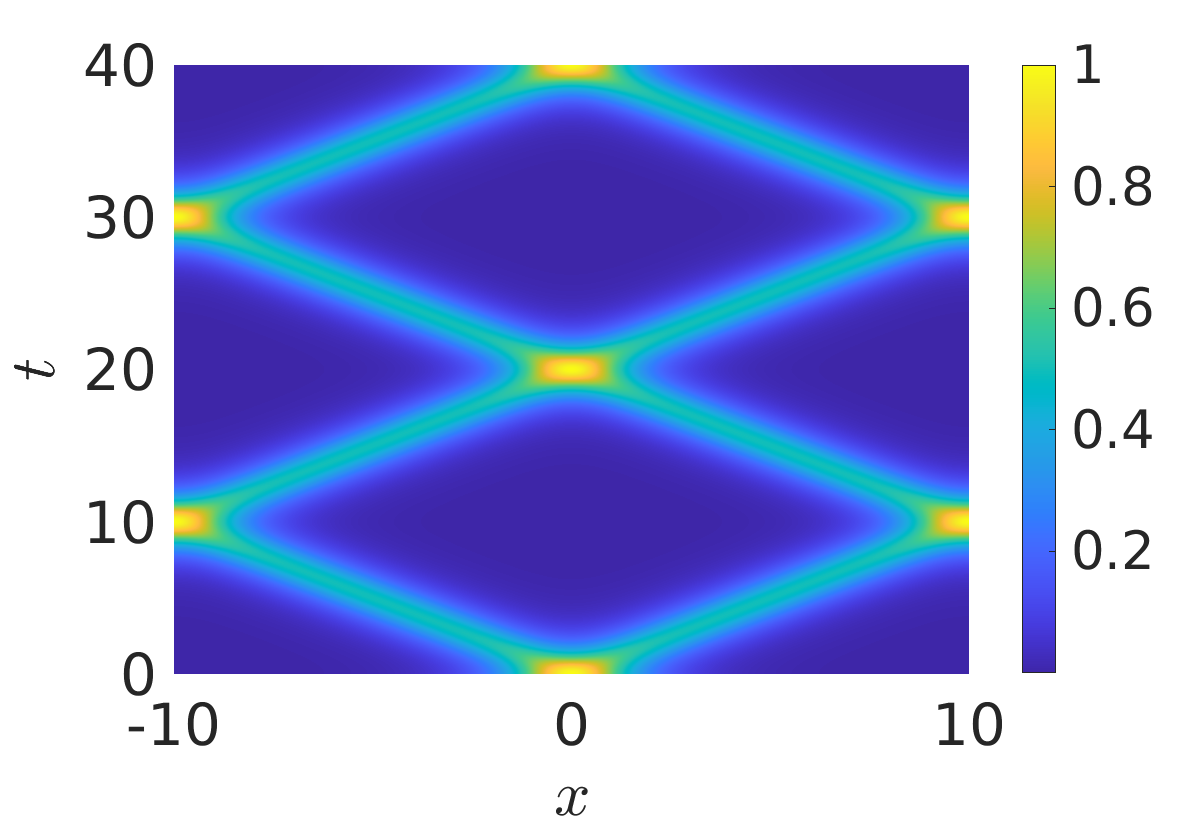} 
	\caption{Linear wave equation:  a comparison between the FOM solution and reduced-order solutions (for order $r = 50$) obtained using the POD-Galerkin and LIGEP-ROM is shown. The FOM solutions is plotted in the left, and the solutions obtained using POD-Galerkin and LIGEP-ROM are plotted in the middle and right, respectively.}
	\label{fig:wave-fmrms}
\end{figure}

Figure~\ref{fig:wave-pol-energy} shows the discrete preserved polarized energies for the FOM \eqref{eqn:wave-fom-pol-ener} and the LIGEP-ROM \eqref{eqn:wave-rom-pol-ener}. We also plot the discrete polarized energy \eqref{eqn:wave-rom-pol-ener} for the POD-Galerkin model in Figure~\ref{fig:wave-pol-energy}, which shows that when the system is stable, the energy remains constant for the POD-Galerkin model; nevertheless, after $T=10$ oscillations in the energy occur, and increasing the order of the POD-Galerkin model does not change the preservation capability of the method.

\begin{figure}[tb]
	\centering
	\includegraphics[width=0.48\linewidth]{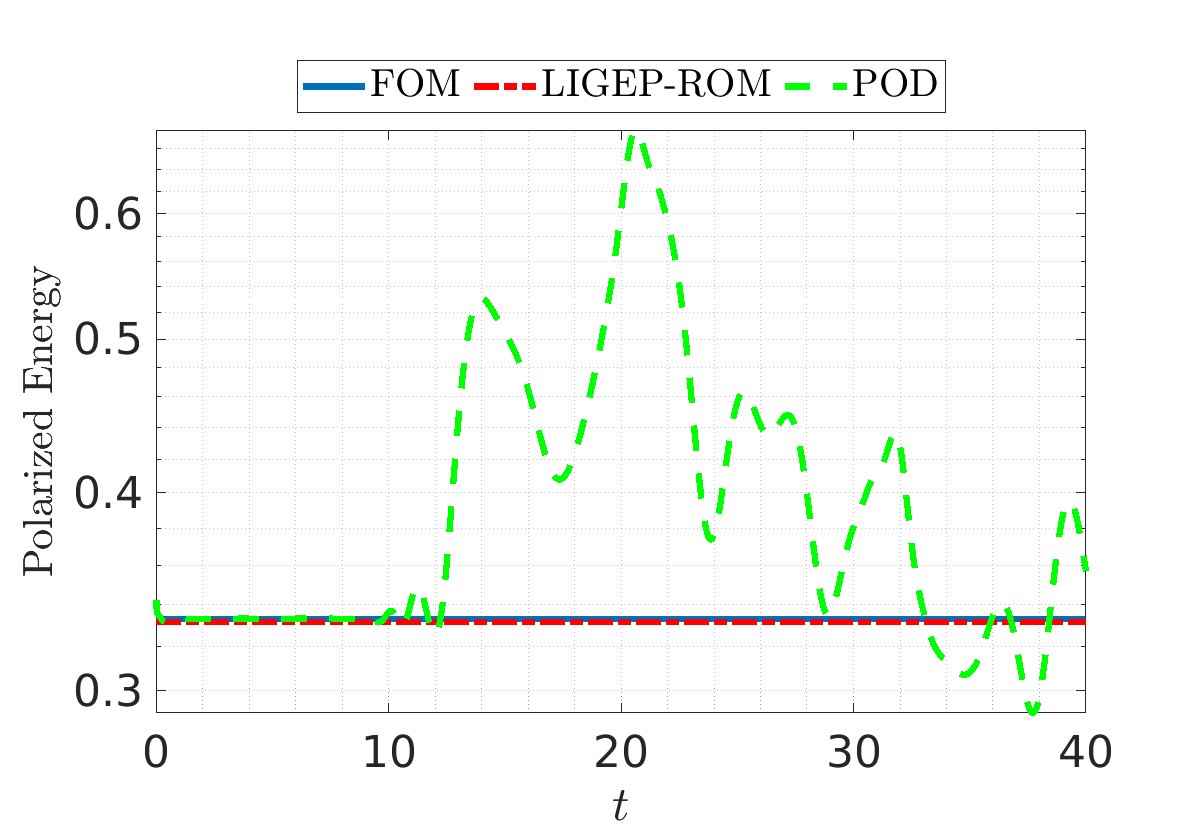}
	\includegraphics[width=0.48\linewidth]{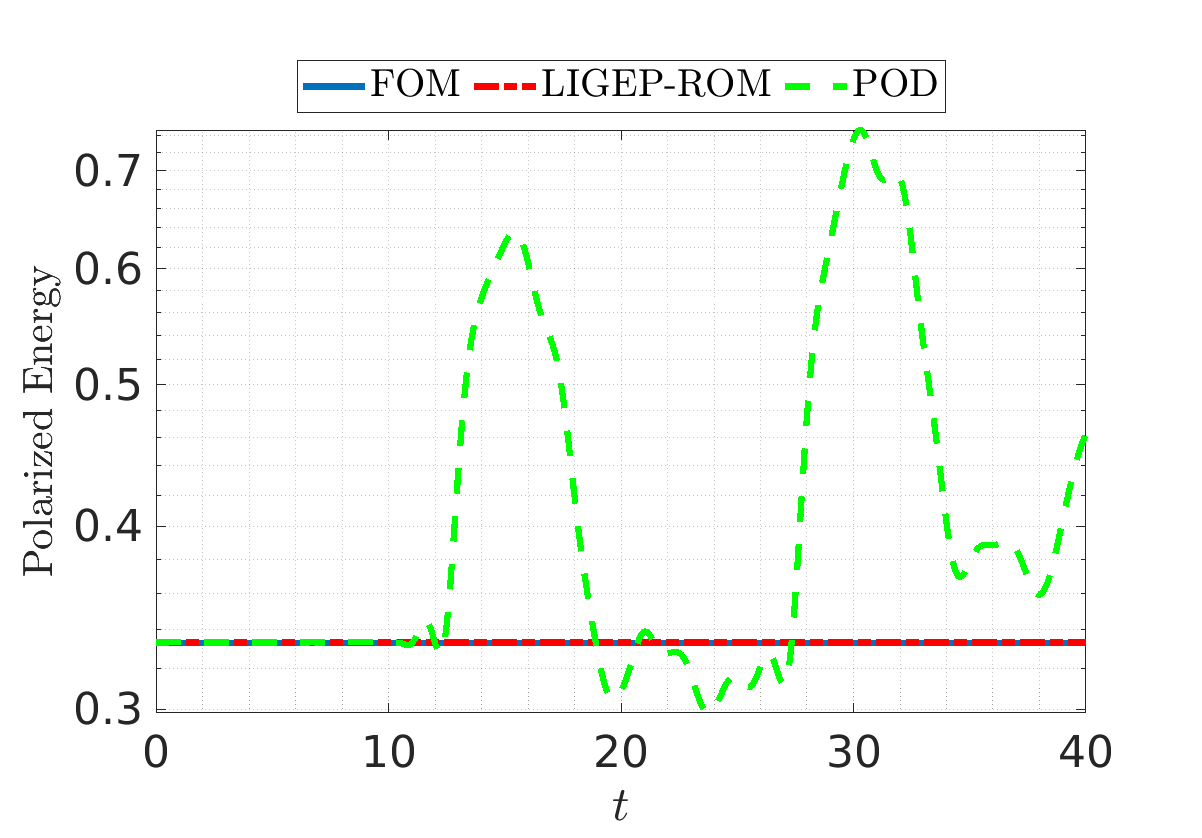}
	\caption{Linear wave equation: polarized discrete global energies of the FOM and ROMs. In the left, the ROMs are of order $r = 20$, and in right, they are of order $r = 50$.}
	\label{fig:wave-pol-energy}
\end{figure}

\begin{figure}[tb]
	\centering
	\includegraphics[width=0.48\linewidth]{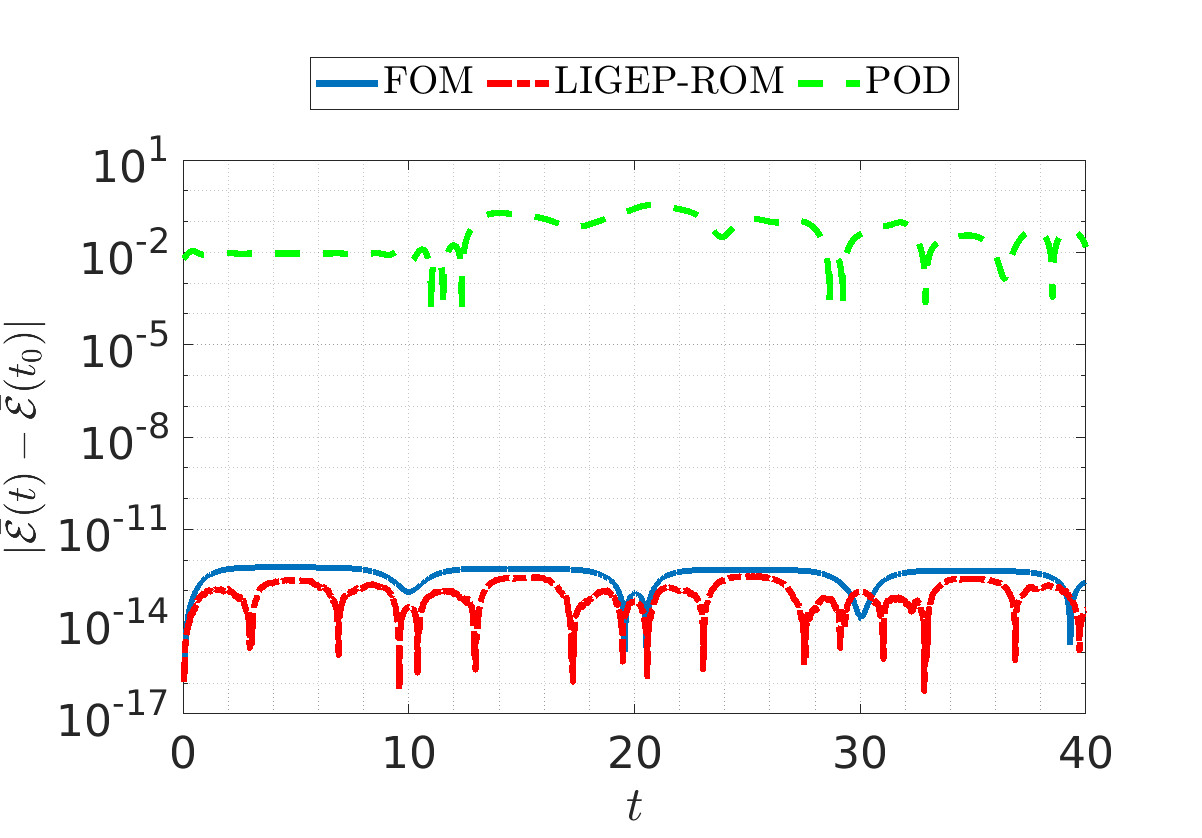}
	\includegraphics[width=0.48\linewidth]{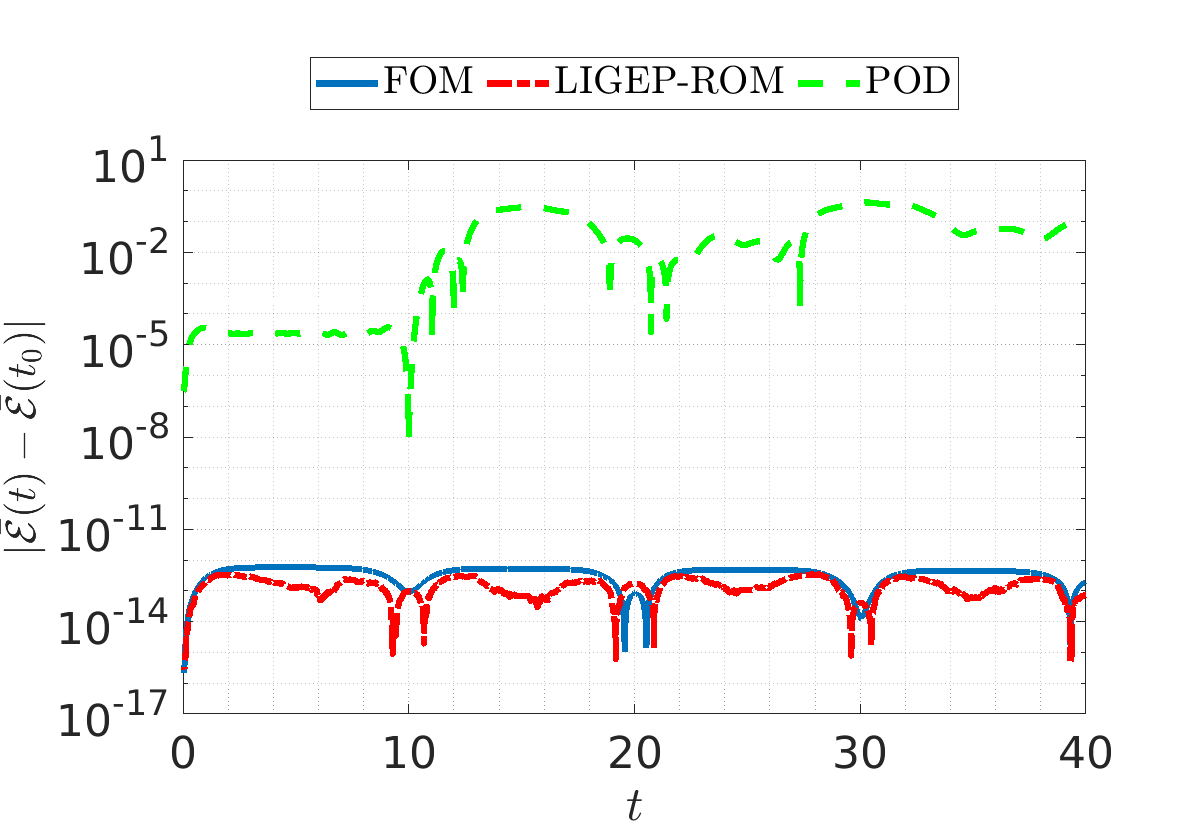}
	\caption{Linear wave equation: polarized discrete global energy errors \eqref{eqn:pole_err} of FOM and ROMs are shown with reduced-order being $ r=20 $ (left) and  $ r=50 $  (right).}
	\label{fig:wave-pol-energy-ac}
\end{figure}

In Figure~\ref{fig:wave-pol-energy-ac}, we examine the performance of the FOM and ROMs in terms of global energy preservation accuracy \eqref{eqn:pole_err}, which verifies the theoretical findings in Theorem~\ref{th:globalenergylaw}. Figure~\ref{fig:wave-pol-energy-ac} confirms that regardless of the order $r$ of a ROM, the polarized reduced-order global energy is preserved with a machine precision accuracy. The figure also shows that by increasing the order of the POD-Galerkin model, the global energy preservation accuracy only increases on the time interval $ [0,10] $ employed for training.

In Figure~\ref{fig:wave-polfr-energy-ac}, we test the ROMs with FOM in terms of accuracy of the approximated polarized energies between FOM and ROMs using \eqref{eqn:pole_fr_err}. Figure~\ref{fig:wave-polfr-energy-ac} shows that by increasing the order of the ROMs, the approximated reduced global energy \eqref{eqn:wave-rom-pol-ener} converges to the polarized global energy \eqref{eqn:wave-fom-pol-ener}, whereas for the POD-Galerkin model, increasing the dimension does not affect the convergence much, at least outside of the training. 

\begin{figure}[tb]
	\centering
	\includegraphics[width=0.48\linewidth]{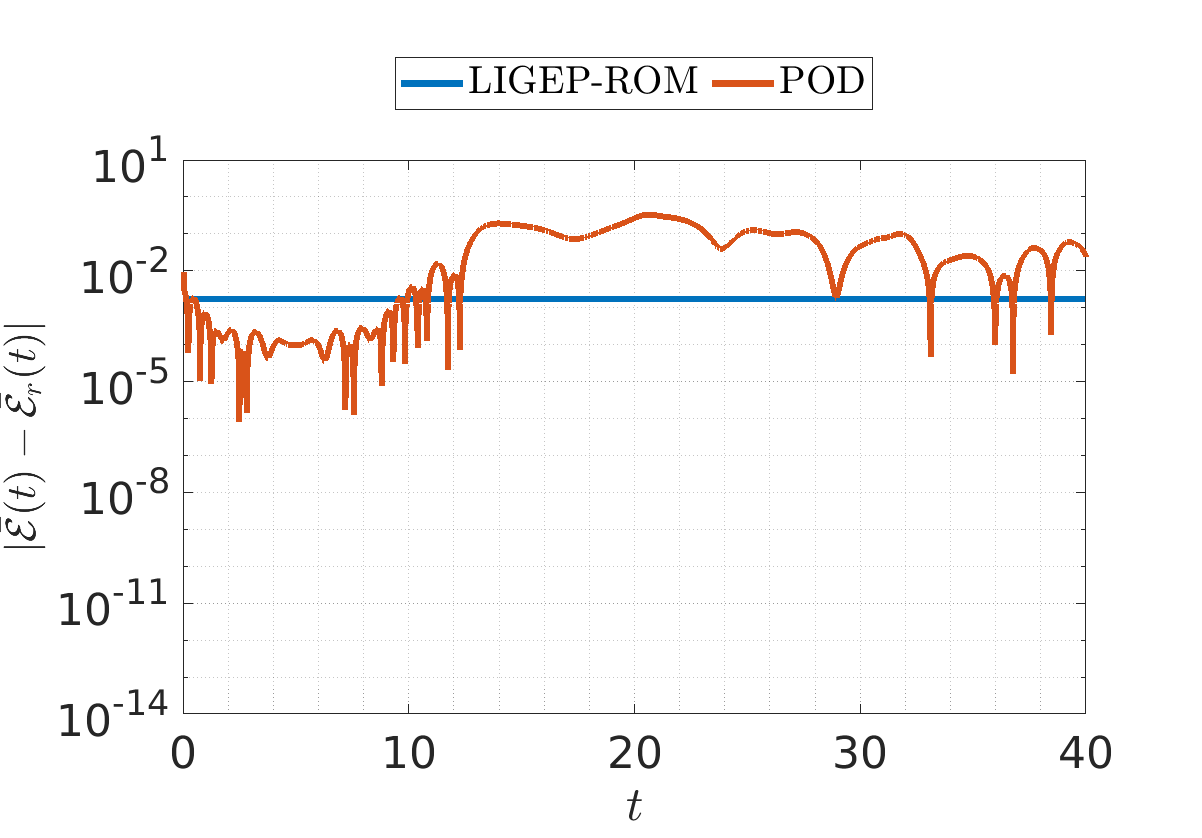}
	\includegraphics[width=0.48\linewidth]{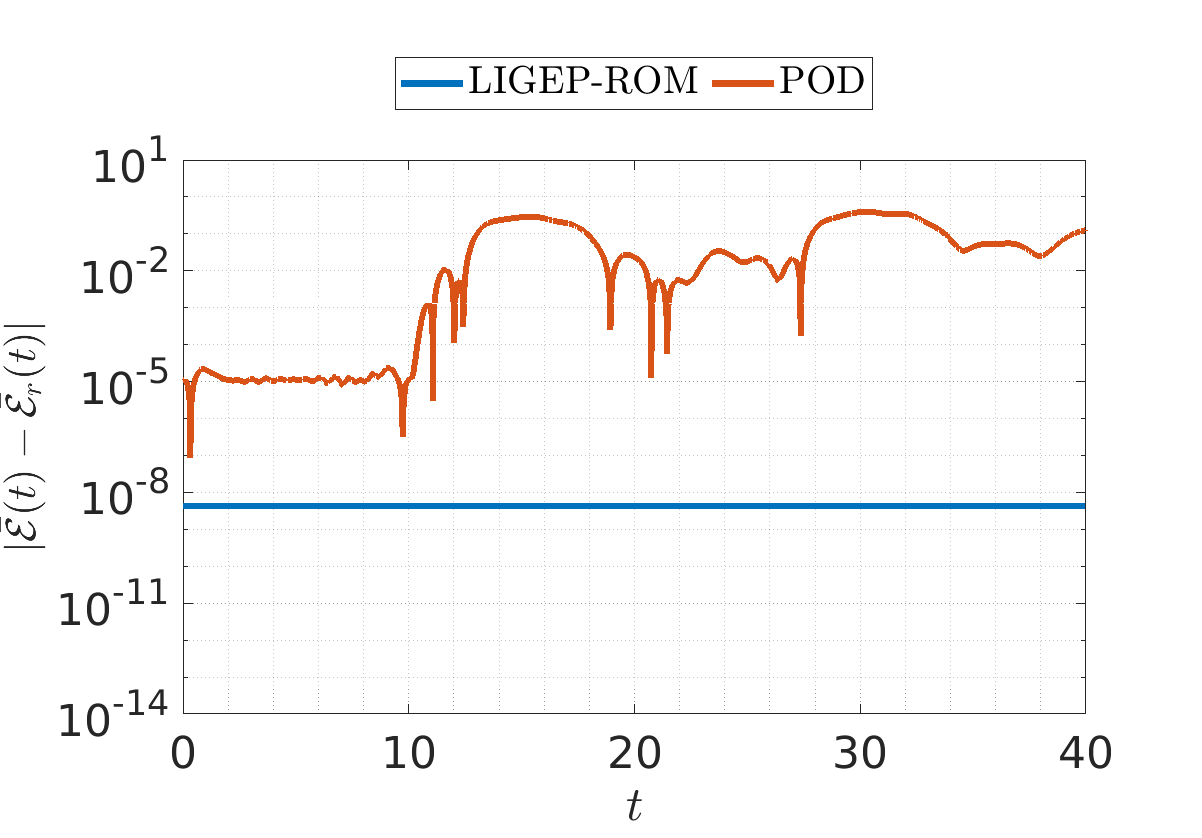}
	\caption{Linear wave equation: polarized discrete global energy errors \eqref{eqn:pole_fr_err} between the FOM and ROMs are shown with reduced-order being $ r=20 $ (left) and  $ r=50 $  (right).}
	\label{fig:wave-polfr-energy-ac}
\end{figure}

\subsection{Korteweg--de Vries equation}\label{subsec:KdV}

For the second test case, we consider the one-dimensional KdV equation, which is frequently used in shallow water waves, internal waves, and plasma physics. The one-dimensional KdV equation has the following form:
\begin{equation}\label{eqn:KdV}
	u_t + \eta u u_x+\gamma^2 u_{xxx}  = 0,
\end{equation}
where $\eta, \gamma \in \mathbb{R}$. By defining the potential $\phi_x=u$, a variable $w=\gamma v_x \phi_t+\frac{\gamma^2u^2}{2}$, and momenta $v=\gamma u_x$, the KdV equation can be written in the form
\begin{equation}\label{eqn:KdV-var}
	\begin{aligned}
		\frac{1}{2}u_t+w_x&=0, \quad&
		-\frac{1}{2}\phi_t-\gamma v_x&=-w+\frac{\eta}{2}u^2,\\
		\gamma u_x&=v,&
		-\phi_x&=-u.
	\end{aligned}
\end{equation}
Using \eqref{eqn:KdV-var}, the KdV equation can be written in the multi-symplectic form \eqref{eqn:Ms_PDEs} with
\begin{equation*}
	K=\begin{bmatrix}
		0       &\frac{1}{2} & 0 & 0 \\
		-\frac{1}{2}       & 0 & 0 & 0 \\
		0      & 0 & 0 & 0 \\
		0      & 0 & 0 & 0  \\
	\end{bmatrix},\qquad
	L=
	\begin{bmatrix}
		0       &0 & 0 &1  \\
		0       & 0 & -\gamma & 0  \\
		0      & \gamma & 0 & 0  \\
		-1      & 0 & 0 & 0 \\
	\end{bmatrix},
\end{equation*}
 $z = (\phi,u,v,w)^\top$, and the Hamiltonian $S(z)$ can be defined as $S(z):=\dfrac{v^2}{2}-uw+\dfrac{\eta u^3}{6}$.
A fully discrete KdV equation with the LIGEP method \eqref{eqn:LIGEP-FOM} reads as

\begin{equation}\label{eqn:KdV-fom-elem}
	\begin{aligned}
		\frac{1}{2}\delta_t u_j^n+\delta_x^{1/2}\mu_tw_j^n &=0,&
		-\frac{1}{2}\delta_t\phi_j^n-\gamma \delta_x^{1/2}\mu_tv_j^n &=-\mu_t w_j^n+\frac{\eta}{2} u_j^n u_j^{n+1},\\
		\gamma \delta_x^{1/2}\mu_tu_j^n &=\mu_t v_j^n,&
		\delta_x^{1/2}\mu_t\phi_j^n &=\mu_t u_j^n.
	\end{aligned}
\end{equation}
After eliminating the auxiliary variables in \eqref{eqn:KdV-fom-elem}, the fully-discrete KdV equation with the LIGEP method becomes
\begin{equation}\label{eqn:KdV_LIGEP}
	\delta_tu_j^n+\frac{\eta}{2}\delta_x^{1/2}(u^n_ju^{n+1}_j)+\gamma^2\mu_t\left(\delta_x^{1/2}\right)^3u^n_j=0.
\end{equation}
The polarised discrete energy preserved by \eqref{eqn:KdV_LIGEP} is
\begin{equation}\label{eqn:KdV-fom-pol-ener}
	\bar{\mathcal{E}}(t_n) = \, \frac {\Delta x}{6}\sum_{j=1}^{N}\left(- \gamma^2(\delta_x^{1/2}u^n_j)^2+2(\delta_x^{1/2}u^n_j)(\delta_x^{1/2} u^{n+1}_j)+\eta(u_j^n)^2 u_j^{n+1} \right).
\end{equation}
For the POD-Galerkin model, we consider the following form of the KdV equation:
$$ u_t = -\eta u u_x-\gamma^2 u_{xxx}.$$
We obtain the POD basis matrix $W\in  \mathbb{R}^{N\times r} $ for the POD-Galerkin model using the following snapshot matrix:
$$\bl S=\left[\bl u(t_1),\ldots,\bl u(t_{N_t})\right]\in \mathbf{R}^{ N\times N_t} ,$$
where the snaphots are generated using the LIGEP model \eqref{eqn:KdV_LIGEP}.
Finally, a semi-discrete POD-Galerkin model is constructed as follows:
$$ \blt u_t =  W^\top \left(-\eta (W \blt u) \circ ( D_x W \blt u)-\gamma^2  D_{xxx} W \blt u\right), $$
where $  D_{xxx} $ corresponds to a second-order central finite difference discretization of the partial derivative~$ \partial_{xxx} $; $ D_{x} $ corresponds to a second-order central finite difference discretization of the partial derivative $ \partial_{x} $, and $ \circ $ denotes the element-wise multiplication of vectors. Then, a fully-discrete POD-Galerkin model is obtained by employing Kahan's method to the semi-discrete POD-Galerkin model.

We obtain the basis for the LIGEP-ROM using the following snapshot matrix:
$$\bl Z=\left[ \boldsymbol{\phi}(t_1),\ldots,\boldsymbol \phi(t_{N_t}),\bl u(t_1),\ldots,\bl u(t_{N_t}),\bl v(t_1),\ldots,\bl v(t_{N_t}),\bl w(t_1),\ldots,\bl w(t_{N_t})\right]\in \mathbb{R}^{N\times 4 N_t}.$$
Here, an approximation of state $ v $ can easily be obtained by using the equation $ v=\gamma u_x $. On the other hand, for the approximation of $ \phi $, we are interested in any $ \phi $ satisfying the equation $ \phi_x=u $. Thus, we consider following time-discrete problem
$$\frac{d}{d x}\boldsymbol \phi(x,t_n)=\bl u(x,t_n), \quad \boldsymbol \phi(x_0,t_n)=0,$$
for obtaining $ \boldsymbol \phi $. For the solution of the above equation, we consider the trapezoid differentiation rule to approximate the state $ \boldsymbol \phi .$
Using the equations in \eqref{eqn:KdV-var}, we obtain the approximation of state $ w $ with following equations
\begin{align*}
	\frac{1}{2}\delta_t u_j^n+\delta_x^{1/2}w_j^n &=0,&
	-\frac{1}{2}\delta_t\phi_j^n-\gamma \delta_x^{1/2}v_j^n &=- w_j^n+\frac{\eta}{2} (u_j^n)^2,\\
	\gamma \delta_x^{1/2}u_j^n &= v_j^n, &
	\delta_x^{1/2}\phi_j^n & = u_j^n.
\end{align*}
To obtain approximation of the state $w_j^n$ using the above equations, we first substitute $ u_j^n=\delta_x^{1/2}\phi_j^n  $ in $ \frac{1}{2}\delta_t u_j^n+\delta_x^{1/2}w_j^n =0 $, and by eliminating the operator $ \delta_x^{1/2} $, we obtain $ \frac{1}{2}\delta_t \phi_j^n +w_j^n =0 $. 

Next, substituting 
$$ \frac{1}{2}\delta_t \phi_j^n =-w_j^n  $$ 
in 
$$ 	-\frac{1}{2}\delta_t\phi_j^n-\gamma \delta_x^{1/2}v_j^n =- w_j^n+\frac{\eta}{2} (u_j^n)^2, $$ 
we obtain  $ w_j^n=\frac{\gamma}{2} \delta_x^{1/2}v_j^n+\frac{\eta}{4}(u_j^n)^2. $ We note that we experimentally observe that to include approximations of the auxiliary variables is important in obtaining a stable ROM.

The ROM of the KdV equation by the LIGEP-ROM \eqref{eqn:LIGEP-ROM} method can be written in a compact form as follow:
\begin{equation}\label{eqn:KdV-LIGEP-ROM}
	\delta_t\blt u^n+\frac{\eta}{2}\tilde{D}_x V^\top (\blh u^n \circ\blh u^{n+1})+\gamma^2\mu_t\tilde{D}_x^3\blt u^n=0.
\end{equation}
The ROM \eqref{eqn:KdV-LIGEP-ROM} preserves the following approximated polarised energy:
\begin{equation}\label{eqn:KdV-rom-pol-ener}
	\begin{split}
		\bar{\mathcal{E}}_r(t_n) = & \, \frac {\Delta x}{6}\sum_{j=1}^{N}\Big(- \gamma^2( V \tilde{D}_x \blt u^n)_j^2+2( V\tilde{D}_x\blt u^n)_j( V \tilde{D}_x \blt u^{n+1})_j+ \eta (\blh u^n)_j^2 \blh u_j^{n+1} \Big).
	\end{split}
\end{equation}
For the KdV equation, we consider the following initial value \cite{eidnes2020}:
 $$u_0(x)= \cos(\pi x),$$ 
 and the parameters $ P=2 $, $\gamma=0.022$, $\eta=1$ on the domain $ [0,P] $ with periodic boundary conditions.  We set the spatial step-size to $ \Delta x=0.001 $, hence, the full order system of dimension is $N = 2000$, and temporal step-size to $ \Delta t=0.01 $. We simulate the FOM~\eqref{eqn:KdV_LIGEP} up to the final time $ T=3 $ to construct the bases for both ROMs.
 
Figure~\ref{fig:KdV-l2} shows the relative state error \eqref{eqn:state-l2err} of the ROMs, where the  POD-Galerkin model exhibits unstable behaviour after $ T=6 $ even after increasing the reduced-order $ r $. For  both orders of the reduced model, i.e., $ r=70 $ and $ r=120 $,  the LIGEP-ROM \eqref{eqn:KdV-LIGEP-ROM} demonstrates a good performance in terms of accuracy and stability.

Due to the unstable behavior of the classical POD-Galerkin model, we only show the solutions of the FOM \eqref{eqn:KdV_LIGEP} and the LIGEP-ROM \eqref{eqn:KdV-LIGEP-ROM} on the time interval $ [0,8] $ in Figure~\ref{fig:KdV}. It shows that the LIGEP-ROM successfully captures the dynamics of the KdV equation with a ROM of order $ r=120 $.

\begin{figure}[tb]
	\centering
	\includegraphics[width=0.48\linewidth]{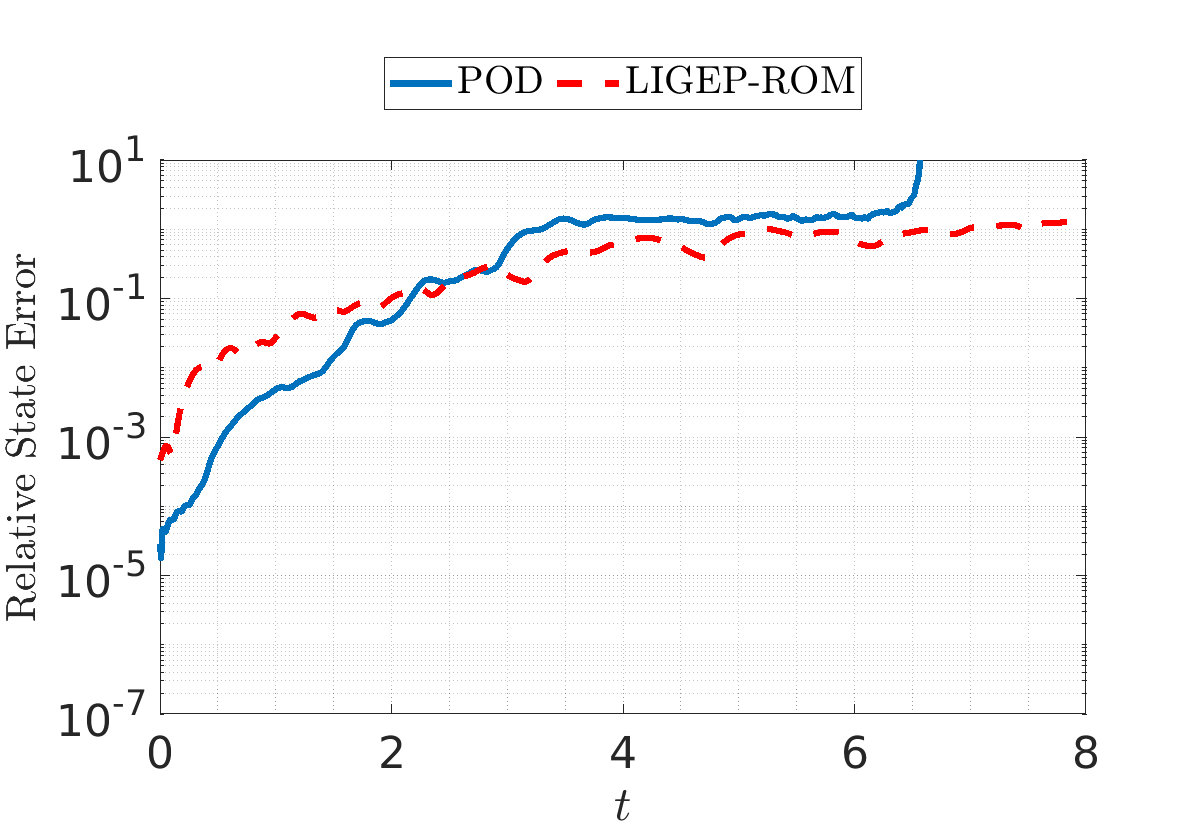}
	\includegraphics[width=0.48\linewidth]{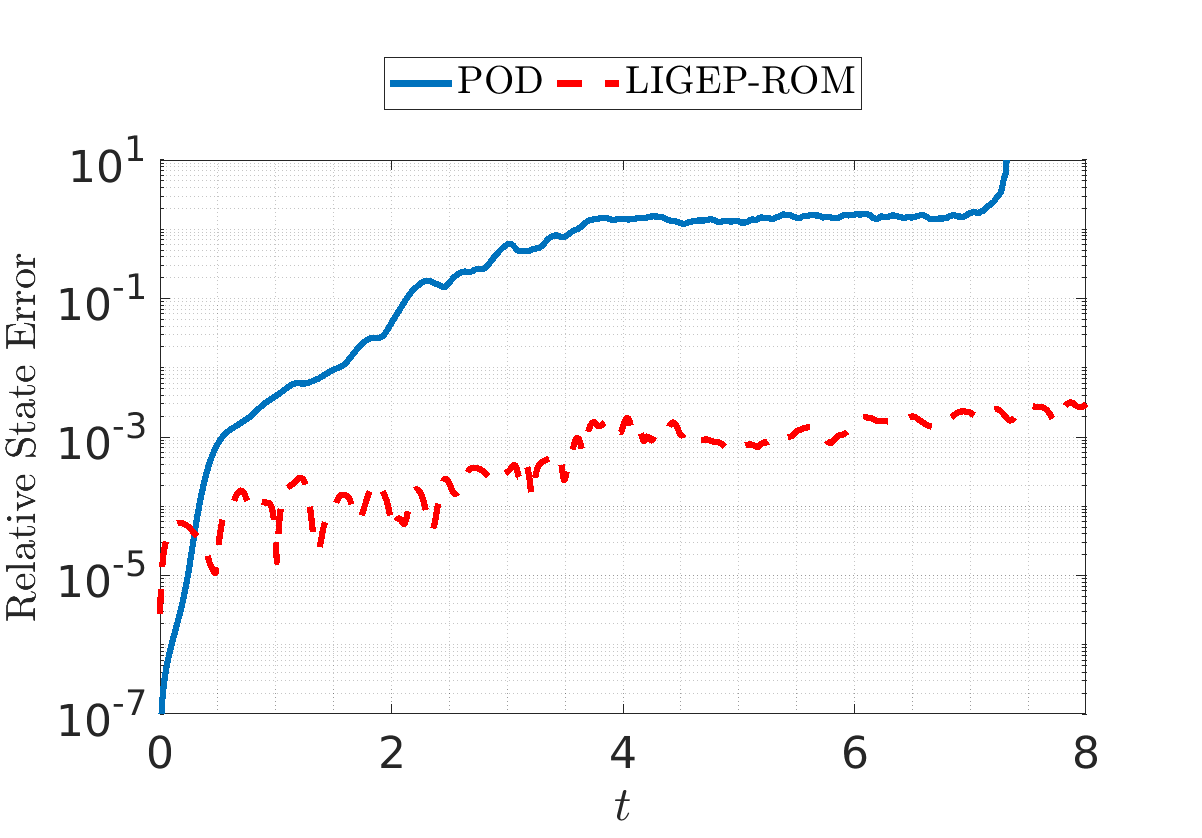}
	\caption{KdV equation: relative state errors for ROMs of order $ r=70 $ and  $ r=120 $ are shown in the left and right, respectively.}
	\label{fig:KdV-l2}
\end{figure}

\begin{figure}[tb]
	\centering
	\includegraphics[width=0.32\linewidth]{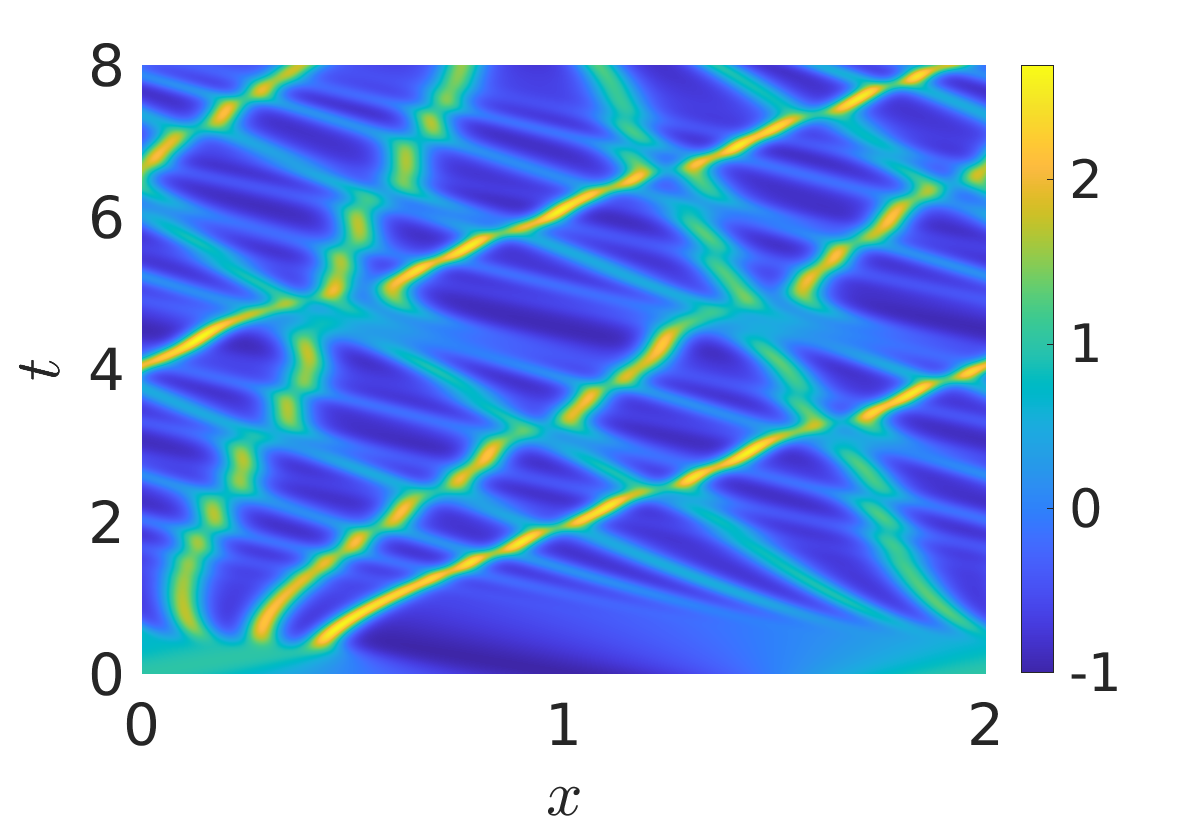} 
	\includegraphics[width=0.32\linewidth]{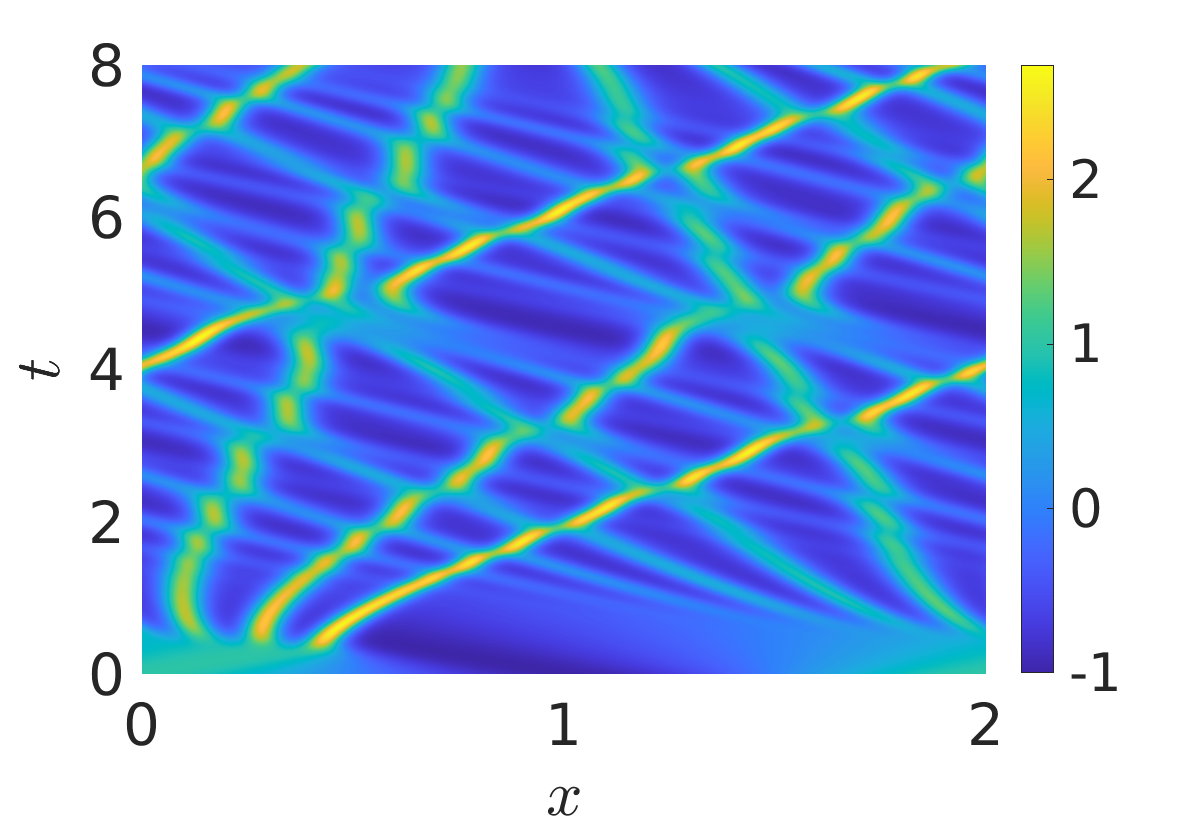} 
	\caption{KdV equation: the FOM solution and LIGEP-ROM solution with order $r = 120$ are shown in the left and right, respectively.}
	\label{fig:KdV}
\end{figure}

Similar to the analysis done in Subsection~\ref{subsec:wave}, we test both ROMs and the FOM in terms of preservation of the discrete polarized energies \eqref{eqn:KdV-rom-pol-ener} and \eqref{eqn:KdV-fom-pol-ener}, respectively. Figure~\ref{fig:KdV-pol-energy} indicates that small drifts over the global energy accuracy yield an unstable POD-Galerkin model regardless of the reduced order. On the other hand, the global energy is preserved for both LIGEP-ROMs of order $ r=70 $ and $ r=120 $. Moreover, we demonstrate the accuracy of the global polarized energy preservation \eqref{eqn:pole_err} of the FOM and ROMs in Figure~\ref{fig:KdV-pol-energy-ac}, which shows that the FOM \eqref{eqn:KdV_LIGEP} and the LIGEP-ROM \eqref{eqn:KdV-LIGEP-ROM} preserve the polarized energy with machine precision accuracy.

\begin{figure}[tb]
	\centering
	\begin{subfigure}{0.48\textwidth}
	\includegraphics[width=1\linewidth]{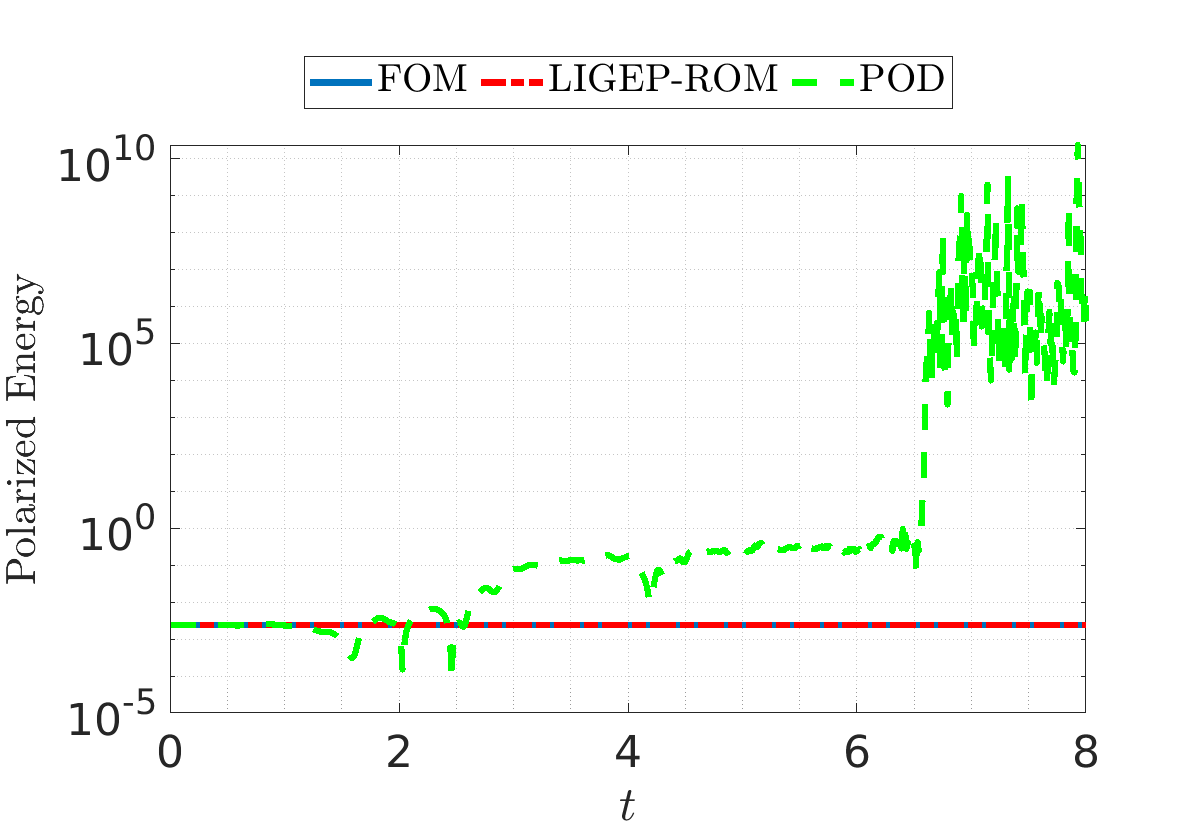}
	\caption{Reduced order $r = 70$.}
	\end{subfigure}
	\begin{subfigure}{0.48\textwidth}
	\includegraphics[width=1\linewidth]{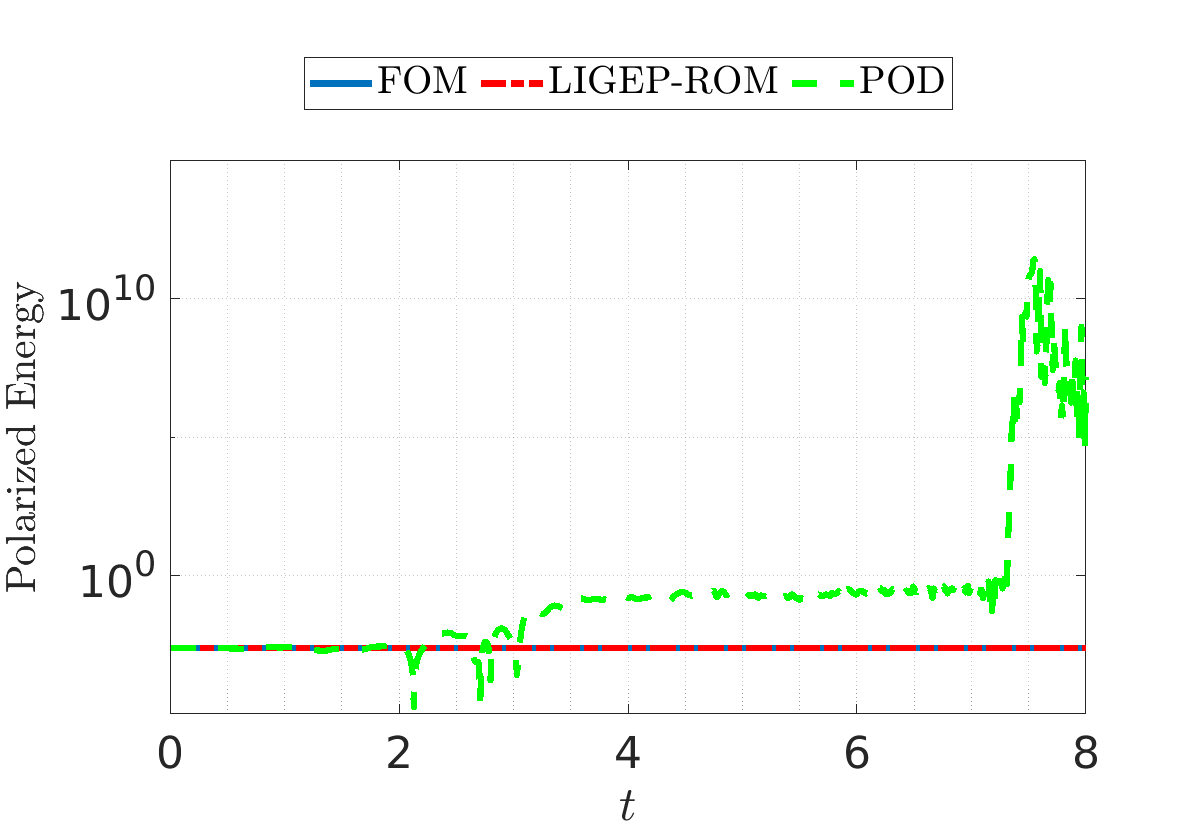}
	\caption{Reduced order $r = 120$.}
\end{subfigure}
	\caption{KdV equation: polarized discrete global energies of the FOM and ROMs obtained using POD-Galerkin and LIGEP-ROM methods.}
	\label{fig:KdV-pol-energy}
\end{figure}

\begin{figure}[tb]
	\centering
	\begin{subfigure}{0.48\textwidth}
		\includegraphics[width=1\linewidth]{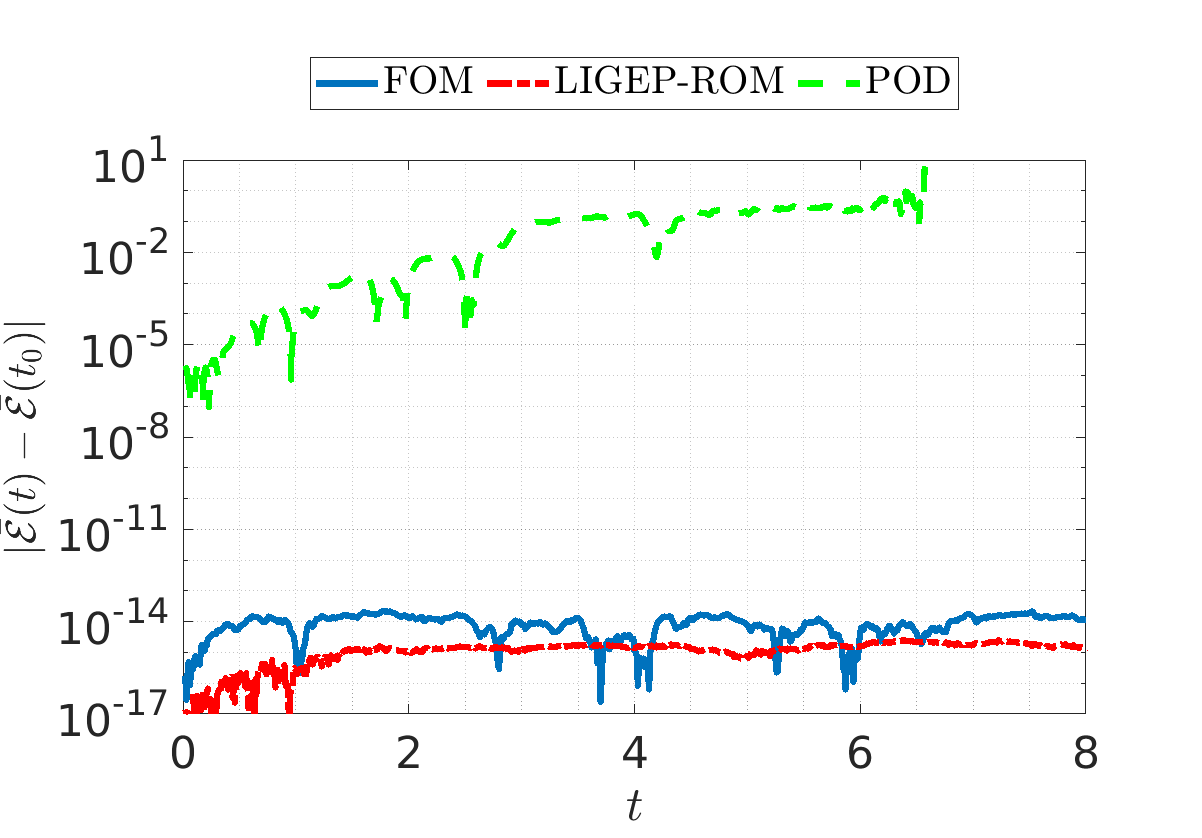}
		\caption{Reduced order $r = 70$.}
	\end{subfigure}
	\begin{subfigure}{0.48\textwidth}
		\includegraphics[width=1\linewidth]{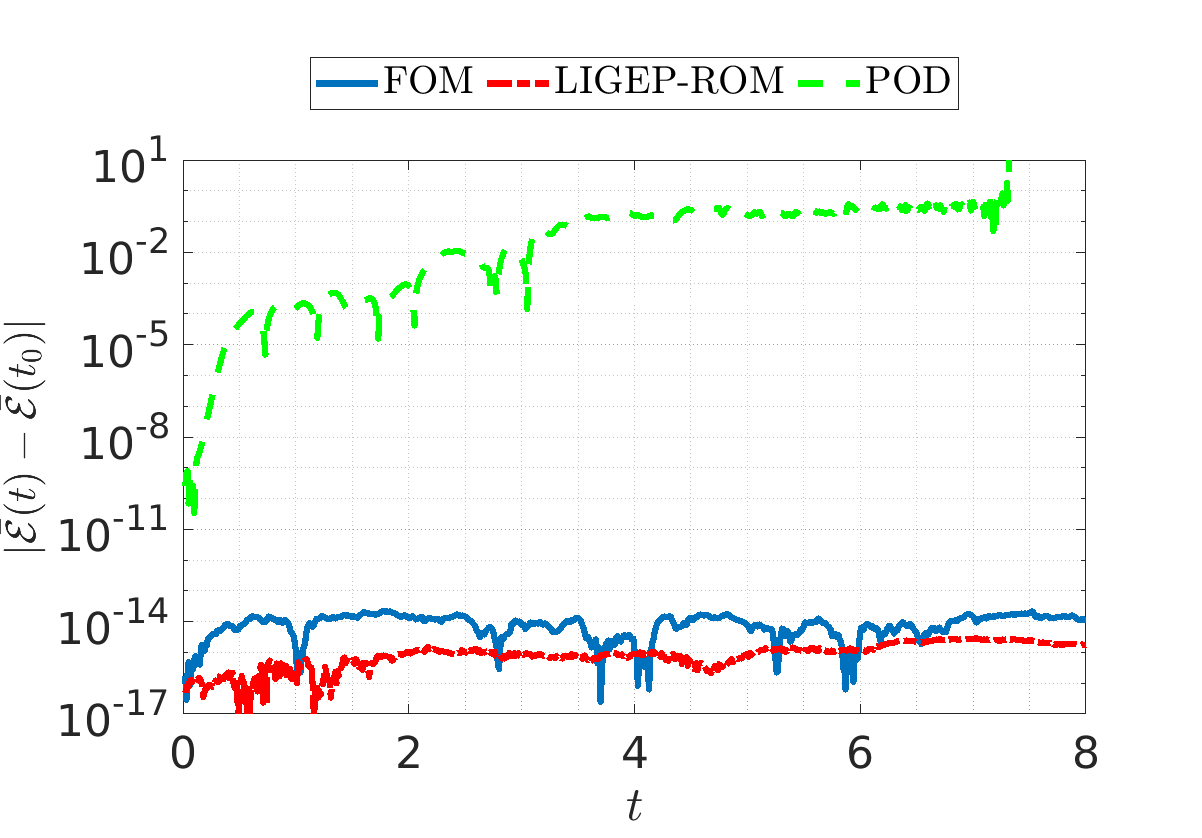}
		\caption{Reduced order $r = 120$.}
	\end{subfigure}
	\caption{KdV equation: polarized discrete global energy errors \eqref{eqn:pole_err} of the FOM and ROMs obtained using POD-Galerkin and LIGEP-ROM methods.}
	\label{fig:KdV-pol-energy-ac}
\end{figure}

For the KdV example, we  lastly examine the error between the approximated polarized energies FOM \eqref{eqn:KdV-fom-pol-ener} and the ROMs \eqref{eqn:KdV-rom-pol-ener} using \eqref{eqn:pole_fr_err}. Figure~\ref{fig:KdV-polfr-energy-ac} shows that by increasing the order of the ROMs from $ r=70 $ to $ r=120 $, the approximated reduced global energy error \eqref{eqn:pole_fr_err} decreases for the LIGEP-ROM \eqref{eqn:KdV-LIGEP-ROM}, where again the order of the classical POD-Galerkin does not significantly affect the error.

\begin{figure}[tb]
	\centering
	\begin{subfigure}{0.48\textwidth}
	\includegraphics[width=1\linewidth]{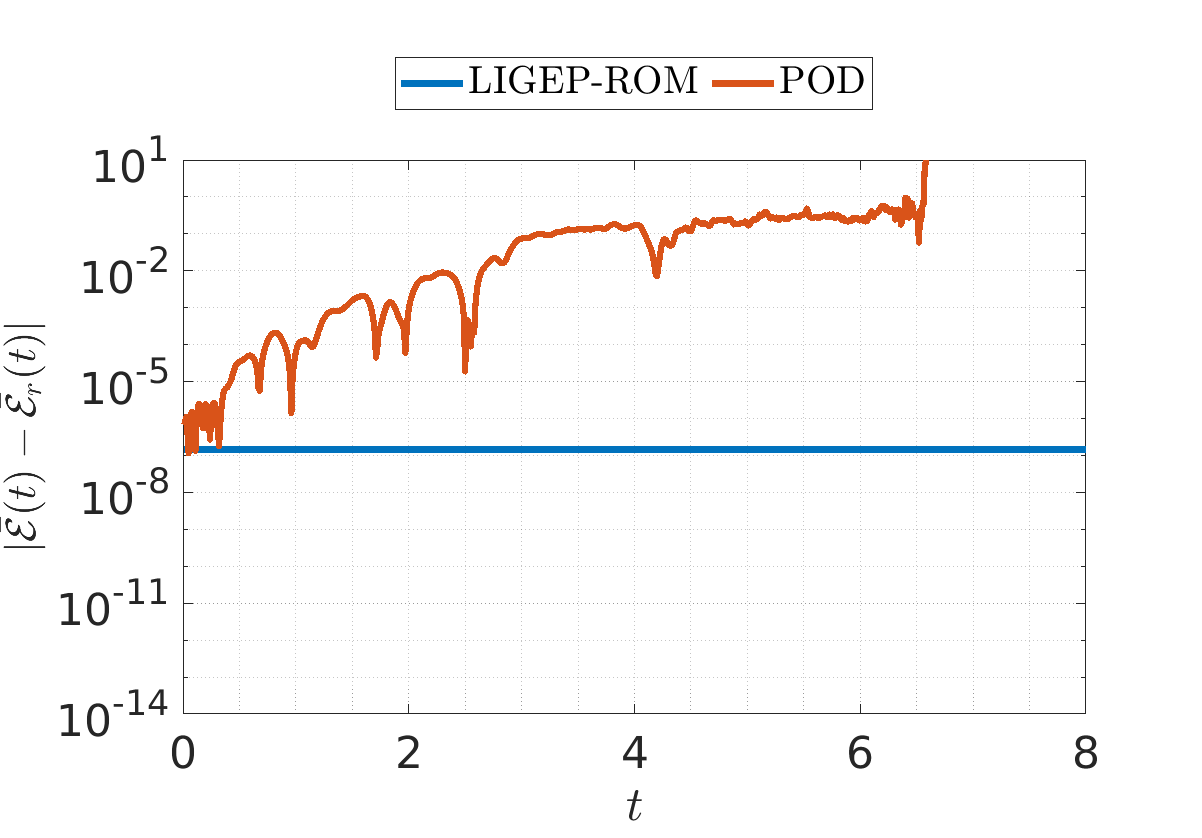}
	\caption{Reduced order $r = 70$.}
\end{subfigure}
\begin{subfigure}{0.48\textwidth}
	\includegraphics[width=1\linewidth]{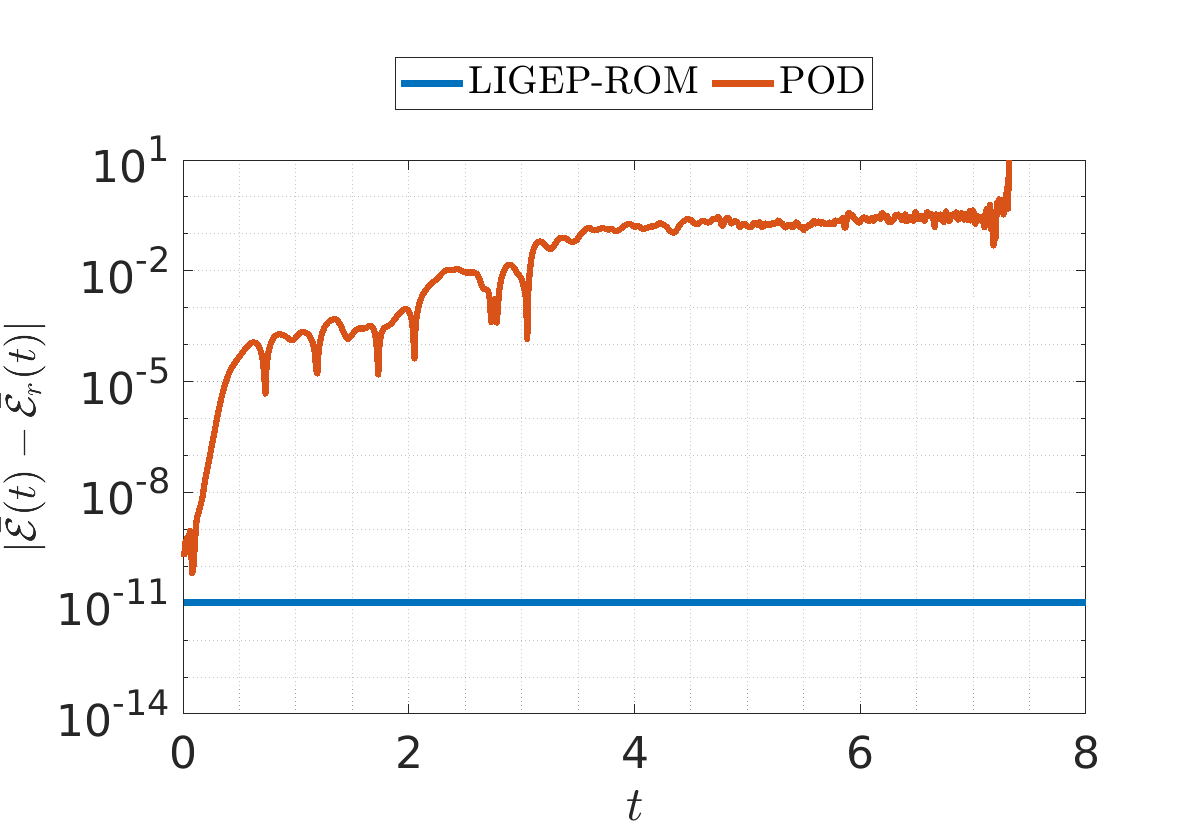}
	\caption{Reduced order $r = 120$.}
\end{subfigure}
	\caption{KdV equation: polarized discrete global energy errors \eqref{eqn:pole_fr_err} between the FOM and ROMs obtained using POD-Galerkin and LIGEP-ROM methods.}
	\label{fig:KdV-polfr-energy-ac}
\end{figure}

\subsection{Camassa-Holm equation}\label{subsec:CH}
In the third test example, we consider the Camassa-Holm (CH) equation
\begin{equation*}
	u_{t} -u_{xxt}+3uu_x-2u_xu_{xx}-uu_{xxx}=0,
\end{equation*}	
which is another model that can be written in the multi-symplectic form \eqref{eqn:Ms_PDEs}. Introducing some new variables, the CH equation can be written as a multi-symplectic Hamiltonian system:
\begin{equation}\label{eqn:CH-fom-elem}
	\begin{aligned}
		\frac{1}{2}\phi_t-\frac{1}{2}\nu_t-v_x&=-w-\frac{3}{2}u^2-\frac{1}{2}\nu^2,\\
		-\frac{1}{2}u_t+w_x&=0,\\
		\phi_x&=u,\\
		-u_x&=-\nu,\\
		-\dfrac{1}{2}u_t&=u\nu-v,\\
	\end{aligned}
\end{equation}
which has the form of \eqref{eqn:Ms_PDEs} with
\begin{equation*}
	z=
	\begin{bmatrix}
		u    \\
		\phi \\
		w    \\
		v    \\
		\nu
	\end{bmatrix},\quad	
	K=
	\begin{bmatrix}
		0       &\frac{1}{2} & 0 & 0 &-\frac{1}{2}\\
		-\frac{1}{2}       & 0 & 0 & 0 &0 \\
		0      & 0 & 0 & 0 &0\\
		0      & 0 & 0 & 0  &0\\
		\frac{1}{2}     & 0 & 0 & 0  &0\\
	\end{bmatrix},\quad
	L=
	\begin{bmatrix}
		0   & 0  & 0 &-1 &0  \\
		0   & 0  & 1 & 0 &0 \\
		0   &-1  & 0 & 0 &0 \\
		1   & 0  & 0 & 0 &0\\
		0   & 0  & 0 & 0 &0\\
	\end{bmatrix},
\end{equation*}
and $S(z)=-wu-\frac{1}{2}u^3-\frac{1}{2}u\nu^2+\nu v$.
A fully discrete multi-symplectic CH equation with the LIGEP \eqref{eqn:LIGEP-FOM} method can be written as:
\begin{equation*}
	\begin{split}
		\frac{1}{2}\delta_t \phi_j^n-\frac{1}{2}\delta_t \nu_j^n-\mu_{t}\delta_x^{1/2}v_j^n&=-\mu_{t}w_j^n-\frac{3}{2}u_j^nu_j^{n+1}-\frac{1}{2}\nu_j^n\nu_j^{n+1},\\
		-\frac{1}{2}\delta_t u_j^n+\mu_{t}\delta_x^{1/2}w_j^n&=0,\\
		\mu_{t}\delta_x^{1/2}\phi_j^n&=\mu_t u_j^n,\\
		-\mu_{t}\delta_x^{1/2}u_j^n&=-\mu_t \nu_j^n,\\
		-\dfrac{1}{2}\delta_t u_j^n&=\frac{1}{2}(u_j^{n+1} \nu_j^n+u_j^n\nu_j^{n+1})-\mu_t v_j^n.\\
	\end{split}
\end{equation*}
Again, after eliminating the additional variables, we obtain the following full-order CH equation of the form:
\begin{equation}\label{eqn:CH-LIGEP-FOM}
	\begin{aligned}
	&	\delta_tu_j^n-\delta_t(\delta_x^{1/2})^2u^n_j-\frac{1}{2}(\delta_x^{1/2})^2((\delta_x^{1/2}u^n_j)u^{n+1}_j)-\frac{1}{2}(\delta_x^{1/2})^2(u^n_j(\delta_x^{1/2}u^{n+1}_j))\\
	&	\hspace{6cm} +\frac{3}{2}\delta_x^{1/2}u^n_ju^{n+1}_j+\frac{1}{2}\delta_x^{1/2}(\delta_x^{1/2}u^n_j)(\delta_x^{1/2}u^{n+1}_j)=0.
	\end{aligned}
\end{equation}
The polarised discrete energy preserved by the fully-discrete model \eqref{eqn:CH-LIGEP-FOM} is
\begin{equation}\label{eqn:CH-fom-pol-ener}
	\begin{split}
		\bar{\mathcal{E}}(t_n) = & \, \frac{\Delta x}{6}\sum_{j=1}^{N}\Big( -3(u_j^n)^2 u_j^{n+1}-(\delta_x^{1/2}u^n_j)^2u_j^{n+1} -2(\delta_x^{1/2}u^n_j)(\delta_x^{1/2}u^{n+1}_j)u^n_j\Big).
	\end{split}
\end{equation}
We construct the POD-Galerkin model by considering the spatial discretization of the CH equations
$$ (1-\partial_{xx})u_{t}=-3uu_x+2u_xu_{xx}+uu_{xxx}=0.$$
The POD basis matrix $W\in  \mathbb{R}^{N\times r} $ for the POD-Galerkin model is obtained from the following snapshot matrix 
$$\bl S=\left[\bl u(t_1),\ldots,\bl u(t_{N_t})\right]\in \mathbf{R}^{ N\times N_t} ,$$
where the snapshots are generated using the LIGEP model \eqref{eqn:CH-LIGEP-FOM}.

A semi-discrete POD-Galerkin model is constructed as follows:
$$ (I_r-W^\top D_{xx}W)\blt u_t =  W^\top \left(-3 (W \blt u) \circ ( D_x W \blt u)+2({D}_x W \blt u) \circ ( D_{xx} W \blt u)+( W \blt u) \circ ( D_{xxx} W \blt u)\right), $$
where $  D_{x},  D_{xx}$ and $  D_{xxx} $ are obtained as discussed in Subsections~\ref{subsec:wave}--\ref{subsec:KdV} . The fully-discrete equations for the classical POD-Galerkin model is obtained by employing Kahan's method \eqref{eqn:Kahan} to the semi-discrete POD-Galerkin model.

To obtain the basis for the LIGEP-ROM, we consider the following snapshot matrix:
$$\bl Z=\left[\bl u(t_1),\ldots,\bl u(t_{N_t}), \boldsymbol{\phi}(t_1),\ldots,\boldsymbol \phi(t_{N_t}),\bl v(t_1),\ldots,\bl v(t_{N_t}),\bl w(t_1),\ldots,\bl w(t_{N_t}),\boldsymbol \nu(t_1),\ldots,\boldsymbol \nu(t_{N_t})\right]\in \mathbb{R}^{N\times 5 N_t}.$$

Here, we approximate the auxiliary variable $ \boldsymbol{\nu}  $ using the equation $ u_x=\nu $. The variable $ \boldsymbol{\phi} $ is approximated as in the previous example using a trapezoid rule by utilizing the equation $ u=\phi_x $. Similarly, we approximate the variable $ u_t $ by a second-order central difference, followed by employing a trapezoid rule to equation $-\frac{1}{2}u_t+w_x=0 $ to obtain the discrete variable $ \bl w. $ Lastly, we approximate the variable $ \bl v $ by using the equation $ v=u\nu +w_x $.

Consequently, the LIGEP-ROM for the CH equation reads as follows:
\begin{equation}\label{eqn:CH-LIGEP-ROM}
	\begin{split}
		\delta_t \blt u^n-\delta_t(\tilde D_x)^2\blt u^n-\frac{1}{2}(\tilde D_x)^2 V^\top\left(( V \tilde D_x \blt u^n)\circ \blh u^{n+1}\right)-\frac{1}{2}(\tilde D_x)^2 V^\top\left(\blh u^n\circ ( V \tilde D_x \blt u^{n+1})\right)\\
		+\frac{3}{2}\tilde D_x V^{T}\left( \blh u^n \circ \blh u^{n+1}\right) +\frac{1}{2}\tilde D_x V^{T}\left( V \tilde D_x \blt u^n\right)\circ\left( V \tilde D_x \blt u^{n+1}\right)=0,
	\end{split}
\end{equation}
where $ \blh u^n $ is the approximation of the state $ \bl u $.
The polarised discrete energy preserved by the global energy preserving ROM \eqref{eqn:CH-LIGEP-ROM} is
\begin{equation}\label{eqn:CH-rom-pol-ener}
	\begin{split}
		\bar{\mathcal{E}}_r(t_n) = & \, \frac{\Delta x}{6}\sum_{j=1}^{N}\Big( -3(\blh u_j^n)^2 \blh u_j^{n+1}-( V \tilde D_x \blt u^n)_j^2 \blh u_j^{n+1} -2( V \tilde D_x \blt u^n)_j( V \tilde D_x \blt u^{n+1})_j \blh u^n_j\Big).
	\end{split}
\end{equation}
Here, we consider the motion of a periodic peaked traveling wave \cite{cai2022} with an initial condition
\begin{equation*}
	u(x,0)=
	\begin{cases}
		\frac{c}{\cosh(a/2)}\cosh(x-x_0), \qquad \qquad \  |x-x_0|\le a/2,\\
		\frac{c}{\cosh(a/2)}\cosh(a-(x-x_0)), \qquad |x-x_0|> a/2,\\
	\end{cases}
\end{equation*}
and a periodic boundary condition. We compute the solution with $ x_0 = 0 $, $  c = 1 $ and $  a = 30 $ on the domain $ [0,a] $. We set the spatial $ \Delta x=0.03 $, hence, the full dimension $N = 1000$, and $ \Delta t=0.005 $. The full-order CH model is simulated up to final time $ T=6 $ for obtaining the basis of both ROMs.

We first test the ROMs in terms of accuracy in Figure~\ref{fig:CH-l2} for orders $ r=70 $ and $ r=120 $. It indicates that the relative state error \eqref{eqn:state-l2err} for  the POD-Galerkin model exhibits an unstable behavior after $ T=6 $. Furthermore, we show the solutions of the models in Figure~\ref{fig:CH}, where, on the contrary to the POD-Galerkin model, the LIGEP-ROM \eqref{eqn:CH-LIGEP-ROM} reflects the FOM \eqref{eqn:CH-LIGEP-FOM} dynamics accurately for order $ r=120 $.

\begin{figure}[tb]
	\centering
	\includegraphics[width=0.48\linewidth]{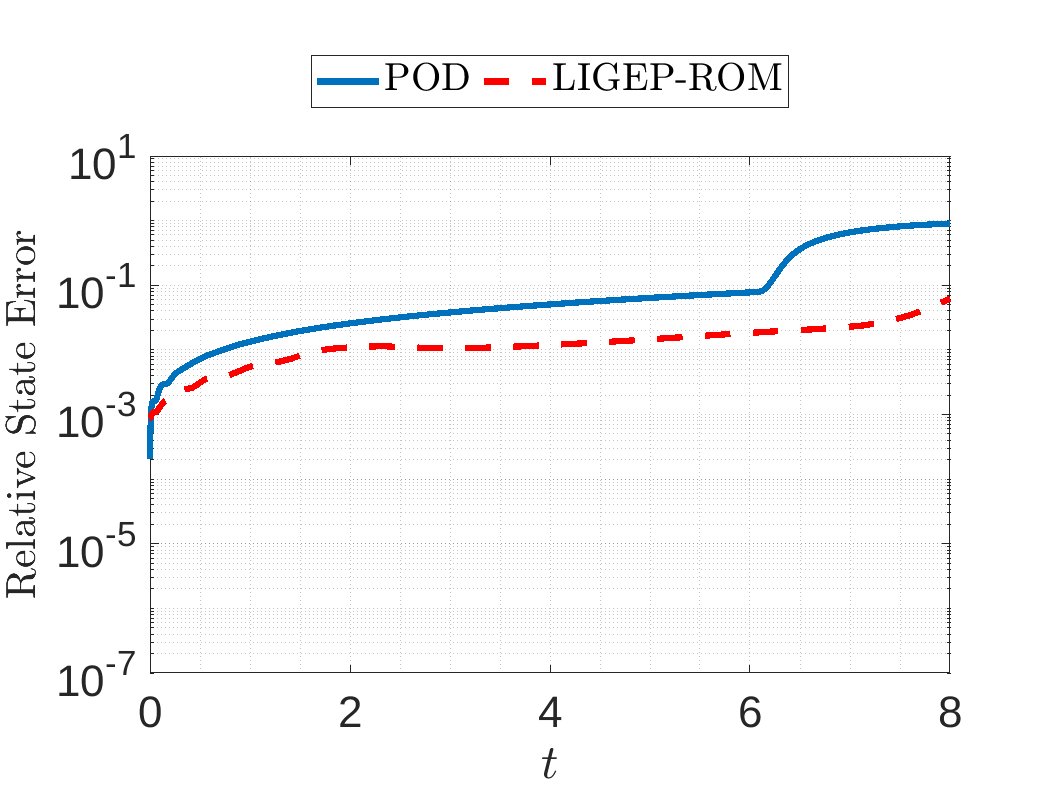}
	\includegraphics[width=0.48\linewidth]{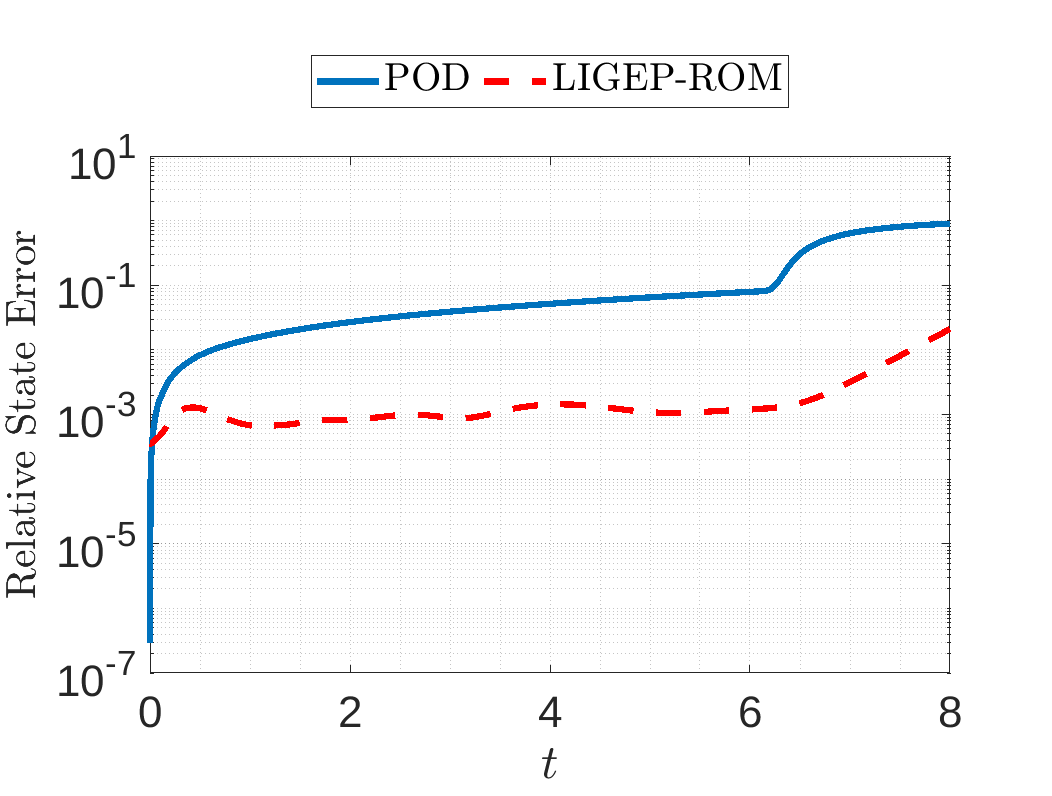}
	\caption{Camassa-Holm equation: relative state errors between the FOM and ROMs of orders $r = 70$ and $120$ are shown in the left and right, respectively.}
	\label{fig:CH-l2}
\end{figure}

\begin{figure}[tb]
	\centering
	\includegraphics[width=0.32\linewidth]{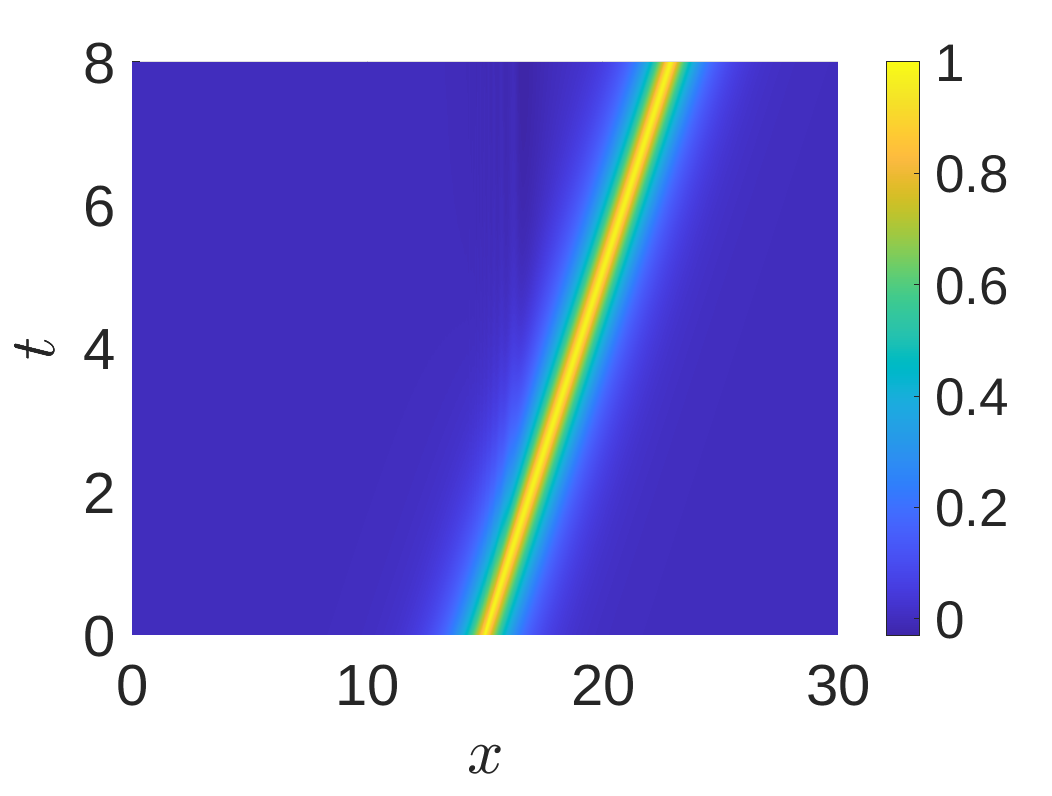} 
	\includegraphics[width=0.32\linewidth]{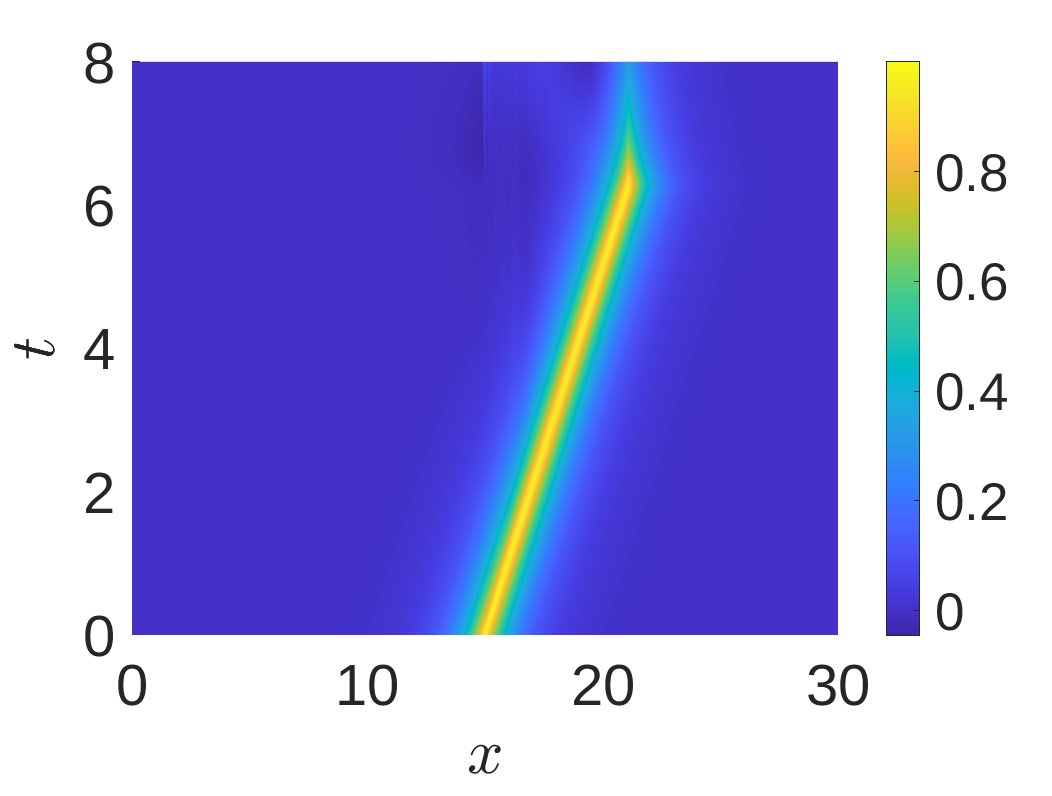} 
	\includegraphics[width=0.32\linewidth]{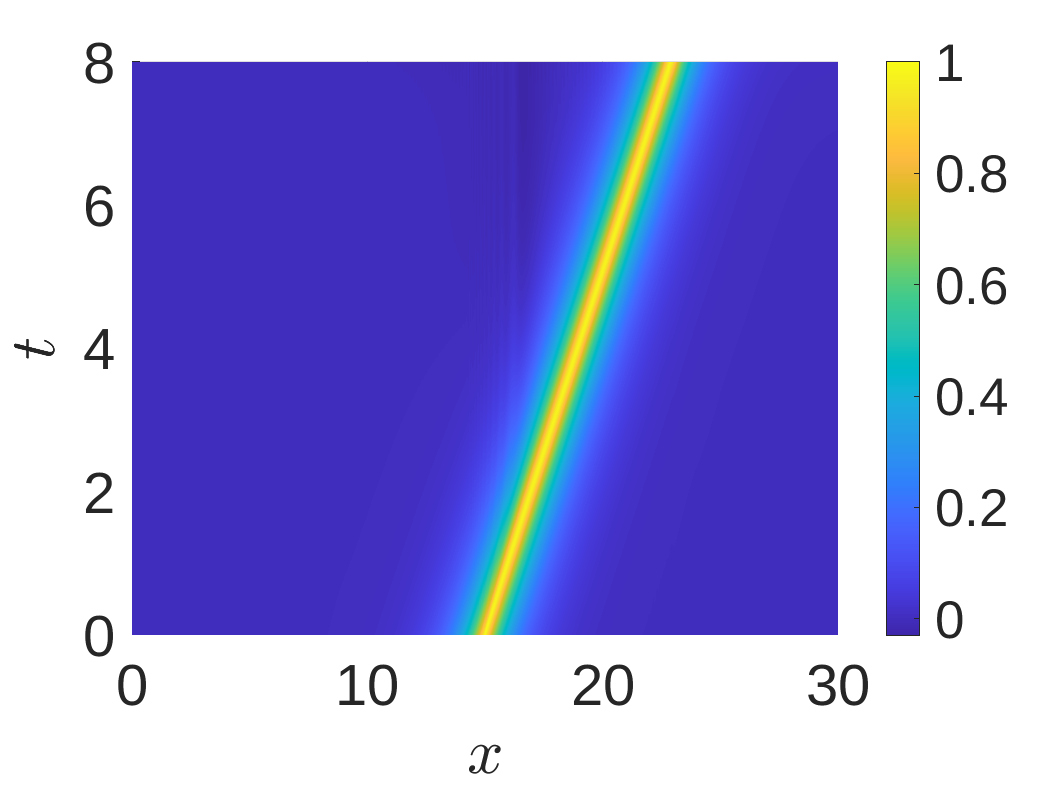} 
	\caption{Camassa-Holm equation:  trajectories obtained using the FOM and ROMs are shown. The order of ROMs is set to $ r=120 $. The FOM, POD-Galerkin ROM, and LIGEP-ROM are shown in the left, middle, and right figure, respectively.}
	\label{fig:CH}
\end{figure}

In Figure~\ref{fig:CH-pol-energy}, we demonstrate the polarized energy \eqref{eqn:CH-rom-pol-ener} preserved by FOM and LIGEP-ROMs for orders $ r=70 $ and $ r=120 $, whereas a dissipative behavior for the POD-Galerkin model is noticed at least outside the training regime. We have observed that dissipation in energy renders the POD-Galerkin model inaccurate. 
Furthermore, we demonstrate the energy preservation accuracy \eqref{eqn:pole_err} of the FOM \eqref{eqn:CH-LIGEP-FOM} and ROMs of order $r=70 $ and $ r=120 $ in Figure~\ref{fig:CH-pol-energy-ac}, where the FOM \eqref{eqn:CH-LIGEP-FOM} and the LIGEP-ROMs \eqref{eqn:CH-LIGEP-ROM} preserve the energy with machine precision accuracy, similar to our two previous examples. Lastly, we show the energy error \eqref{eqn:pole_fr_err} between the FOM \eqref{eqn:CH-LIGEP-FOM} and ROMs of order $r=70 $ and $ r=120 $ in Figure~\ref{fig:CH-pol-energy-fr}, which shows that the approximated energy \eqref{eqn:CH-rom-pol-ener} is converging slowly to the true energy \eqref{eqn:CH-fom-pol-ener}.

\begin{figure}[tb]
	\centering
	\begin{subfigure}{0.48\textwidth}
	\includegraphics[width=1\linewidth]{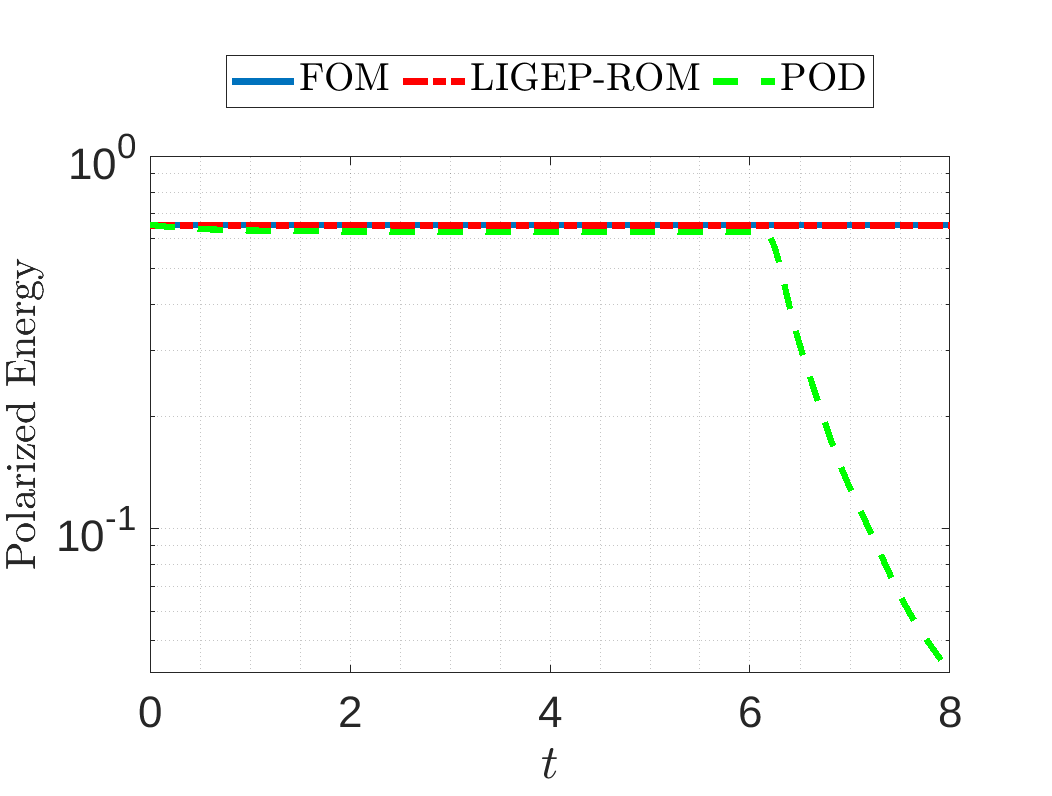}
	\caption{Reduced order $r = 70$.}
	\end{subfigure}
	\begin{subfigure}{0.48\textwidth}
	\includegraphics[width=1\linewidth]{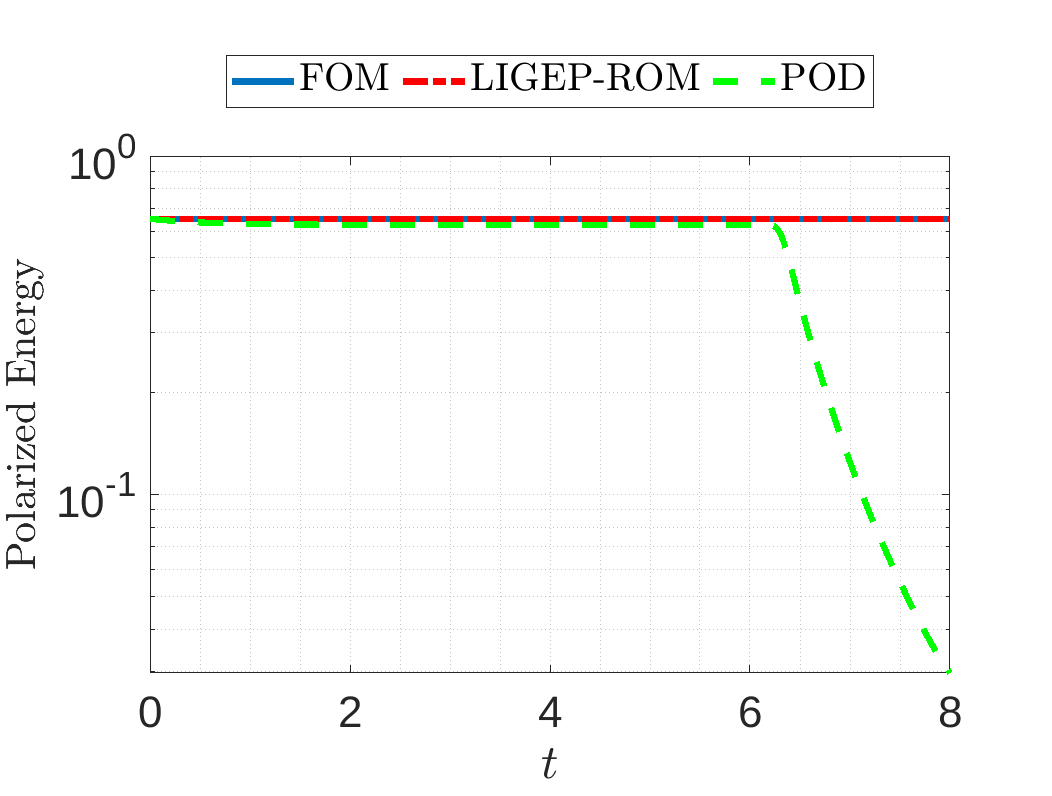}
	\caption{Reduced order $r = 120$.}
\end{subfigure}
	\caption{Camassa-Holm equation: polarized discrete global energies of the FOM and ROMs are shown.}
	\label{fig:CH-pol-energy}
\end{figure}

\begin{figure}[tb]
	\centering
	\begin{subfigure}{0.48\textwidth}
	\includegraphics[width=1\linewidth]{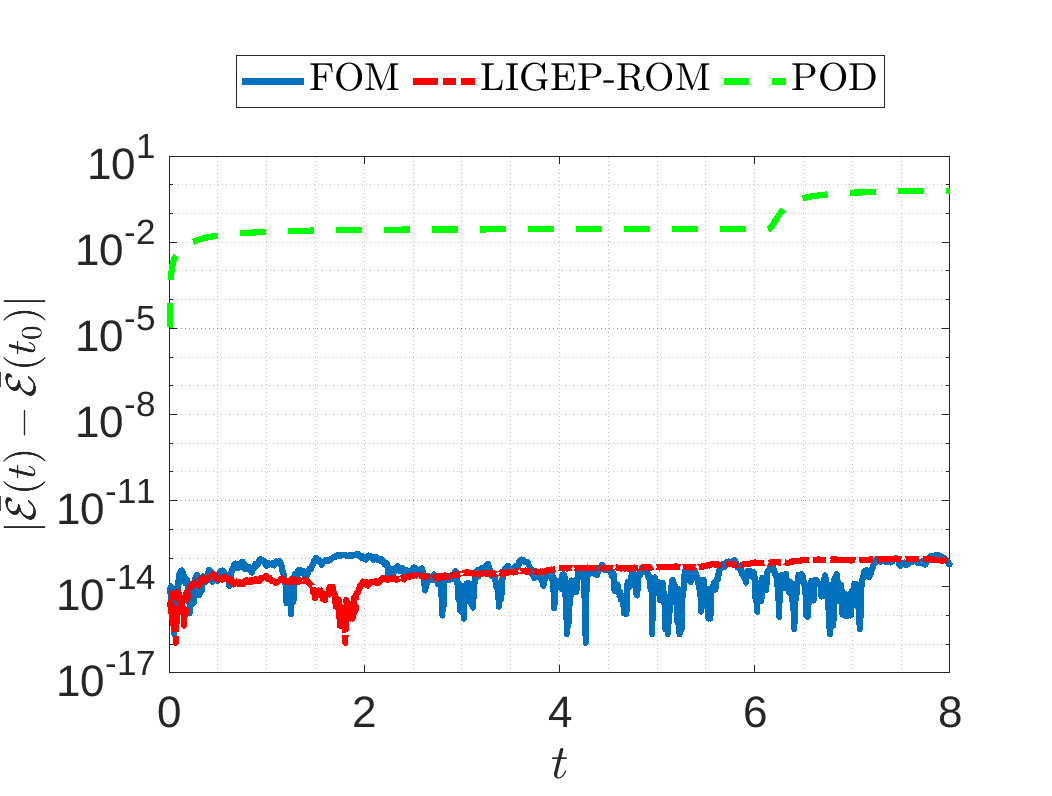}
	\caption{Reduced order $r = 70$.}
\end{subfigure}
\begin{subfigure}{0.48\textwidth}
	\includegraphics[width=1\linewidth]{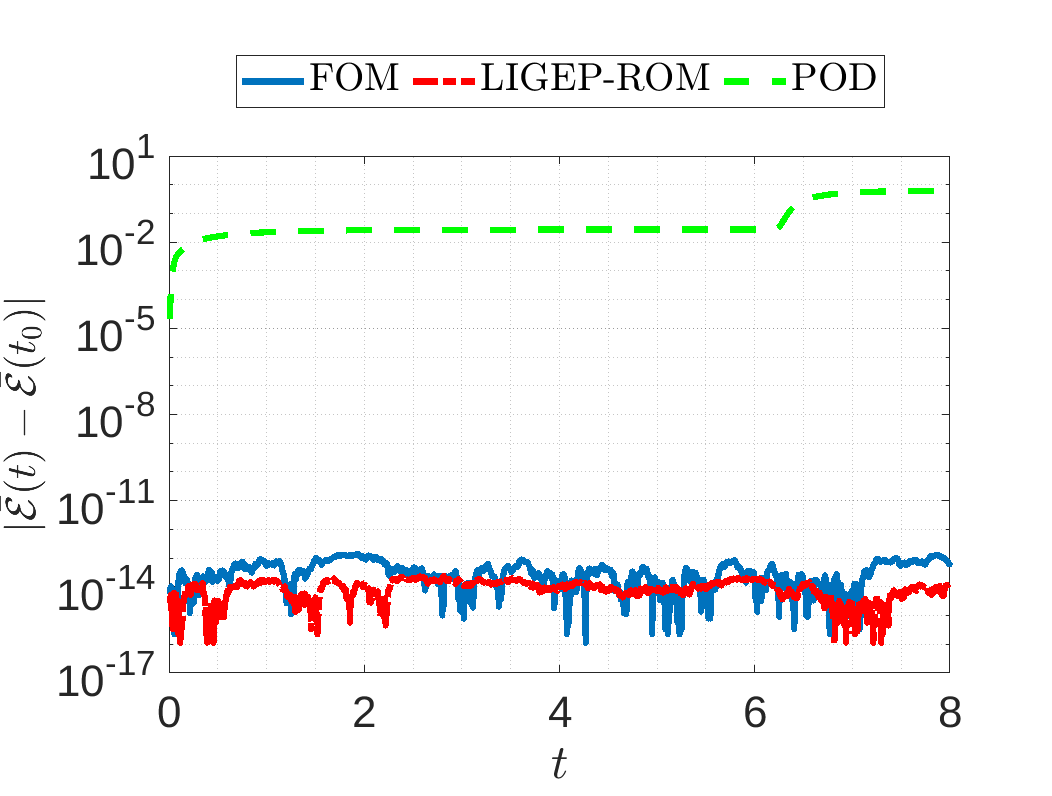}
	\caption{Reduced order $r = 120$.}
\end{subfigure}
	\caption{Camassa-Holm equation:  polarized discrete global energy errors \eqref{eqn:pole_err} of the FOM and ROMs obtained using the POD-Galerkin and LIGEP-ROM methods.}
	\label{fig:CH-pol-energy-ac}
\end{figure}

\begin{figure}[tb]
	\centering
	\begin{subfigure}{0.48\textwidth}
	\includegraphics[width=1\linewidth]{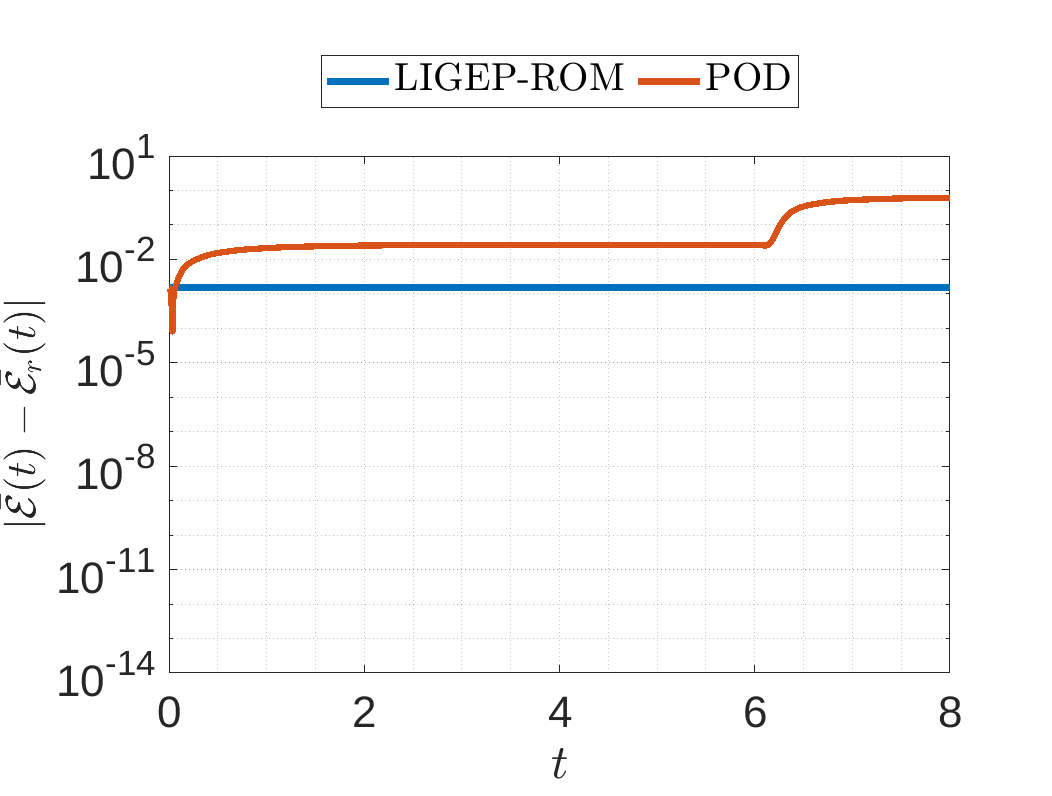}
	\caption{Reduced order $r = 70$.}
\end{subfigure}
\begin{subfigure}{0.48\textwidth}
	\includegraphics[width=1\linewidth]{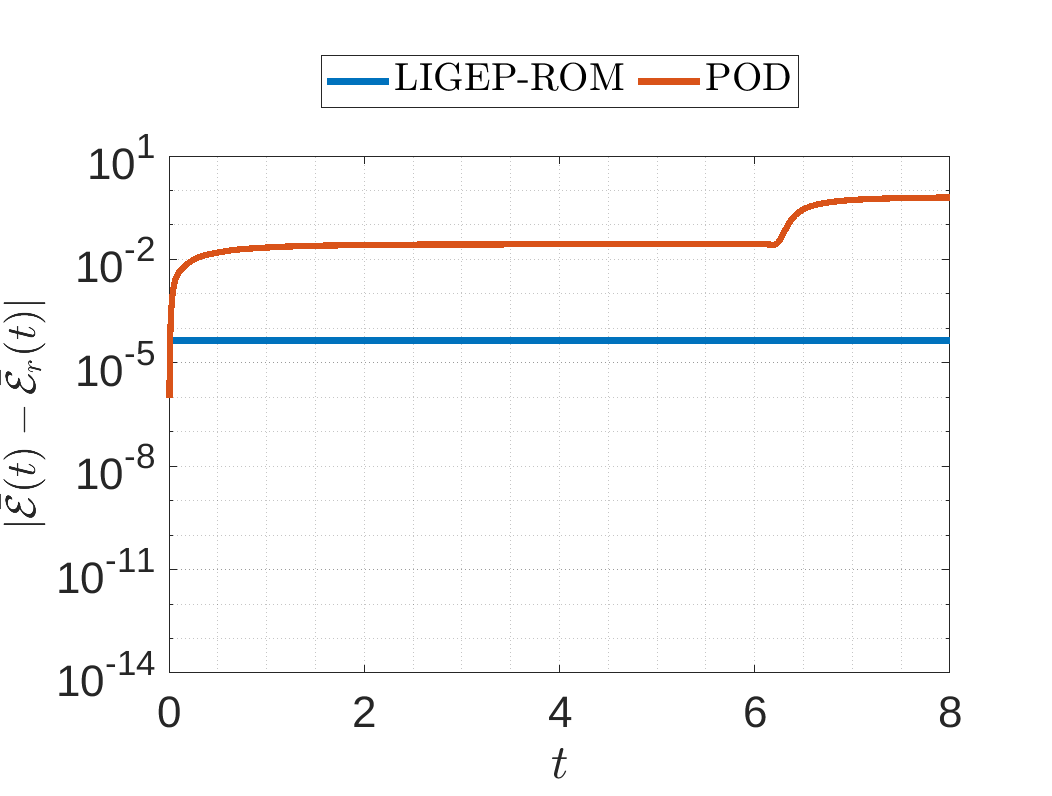}
	\caption{Reduced order $r = 120$.}
\end{subfigure}

	\caption{Camassa-Holm equation:  polarized discrete global energy errors \eqref{eqn:pole_fr_err} between the FOM and ROMs obtained using the POD-Galerkin and LIGEP-ROM methods.}
	\label{fig:CH-pol-energy-fr}
\end{figure}

 \section{Conclusions}\label{sec:conc}

 In this paper, we have proposed the construction of linearly implicit global energy preserving reduced-order models (LIGEP-ROMs) for cubic-Hamiltonian systems, which exploit the multi-symplectic structure of PDEs. 
We have proven that the LIGEP-ROMs preserve global energy, which is also illustrated in our numerical examples. 
Furthermore, we have shown that the spatially-discrete equations of the LIGEP-ROMs satisfy a the spatially-discrete local energy conversation law. We have demonstrated the efficiency of the proposed approach with several numerical examples and have compared it with the classical POD-Galerkin model. This illustrates that energy-preserving reduced-order models are robust and suitable for long-time integration and for predictions outside the training data. 

In our future work, we would like to investigate how to construct global energy preserving reduced-order models directly from data. Moreover, in our examples, we have observed a slow decay of singular values of the snapshot matrix. Therefore, it would be interesting to explore the possibility of using non-projection methods such as autoencoders to further reduce the intrinsic dimensional of a reduced-order model. 

\section*{Acknowledgment}

\section*{Funding Statement}
This work supported by the German Research Foundation (DFG) Research Training Group 2297 ``MathCoRe'', Magdeburg. 

\section*{Data Availability}

The data that support the findings of this study are available from the corresponding author upon reasonable request.

\addcontentsline{toc}{section}{References}
\bibliographystyle{ieeetr}
\bibliography{ref}

\end{document}